\documentclass{article}

\usepackage{amscd,amsmath,amssymb,amsthm,xcolor,mathrsfs,dsfont}
\newcommand{\eps}{\varepsilon}

\newcommand{\R}{\mathbb R}
\newcommand{\N}{\mathbb N}

\newcommand{\then}{\Longrightarrow}
\newcommand{\J}{{\cal J}}

\newcommand\e{{\rm e}}

\DeclareMathOperator*{\esssup}{ess\ sup}
\DeclareMathOperator*{\essinf}{ess\ inf}
\DeclareMathOperator{\supp}{supp}
\newtheorem{corollary}{Corollary}[section]
\newtheorem{theorem}[corollary]{Theorem}
\newtheorem{lemma}[corollary]{Lemma}
\newtheorem{proposition}[corollary]{Proposition}

\theoremstyle{definition}
\newtheorem{definition}[corollary]{Definition}
\newtheorem{remark}[corollary]{Remark}

\numberwithin{equation}{section}
\begin{document}

\title{{\bf Bounded solutions for quasilinear 
modified Schrödinger equations\footnote{The research that led 
to the present paper was partially supported 
by MIUR--PRIN project ``Qualitative and quantitative aspects of nonlinear PDEs''
 (2017JPCAPN\underline{\ }005) and by {\sl Fondi di Ricerca di Ateneo} 2017/18
 ``Problemi differenziali non lineari''. All the authors are members of the Research Group INdAM--GNAMPA.}}}

\author{Anna Maria Candela, Addolorata Salvatore and Caterina Sportelli\\
{\small Dipartimento di Matematica}\\
{\small Universit\`a degli Studi di Bari Aldo Moro} \\
{\small Via E. Orabona 4, 70125 Bari, Italy}\\
{\small \it annamaria.candela@uniba.it}\\
{\small \it addolorata.salvatore@uniba.it}\\
{\small \it caterina.sportelli@uniba.it}}
\date{}

\maketitle

\begin{abstract}
In this paper we establish a new existence result for the quasilinear elliptic problem 
\[
-{\rm div}(A(x,u)|\nabla u|^{p-2}\nabla u) +\frac1p A_t(x,u)|\nabla u|^p 
+ V(x)|u|^{p-2} u = g(x,u)\quad\mbox{ in } \R^N,
\]
with $N\ge 2$, $p>1$ and $V:\R^N\to\R$ suitable measurable positive function,
which generalizes the modified Schrödinger equation. 
Here, we suppose that $A:\R^N\times\R\rightarrow\R$ is a $\mathcal{C}^{1}$--Caratheodory function  
such that $A_t(x,t) = \frac{\partial A}{\partial t} (x,t)$ and a given Carath\'eodory function 
$g:\R^N\times\R\rightarrow\R$ has a subcritical growth and satisfies the Ambrosetti--Rabinowitz condition. 

Since the coefficient of the principal part depends also on the solution itself, 
we study the interaction of two different norms in a suitable Banach space so
to obtain a ``good'' variational approach.
Thus, by means of approximation arguments on bounded sets 
we can state the existence of a nontrivial weak bounded solution. 
\end{abstract}



\section{Introduction} \label{secintroduction}

In this paper we investigate the existence of weak bounded solutions 
for the quasilinear elliptic problem
\begin{equation} \label{problem}
-{\rm div}(A(x,u)|\nabla u|^{p-2}\nabla u) +\frac1p A_t(x,u)|\nabla u|^p + V(x)|u|^{p-2} u = g(x,u)\quad\mbox{ in } \R^N,
\end{equation}
with $N\ge 2$, $p>1$, where $A:\R^N\times\R\to\R$ is a $\mathcal{C}^1$--Carath\'eodory function with partial derivative 
$A_t(x,t) = \frac{\partial A}{\partial t} (x,t)$, potential $V:\R^N\to\R$ is a suitable measurable function 
and $g:\R^N\times\R\to\R$ is a given Carath\'eodory function.

Special examples of problem \eqref{problem} are related to the existence of solitary waves
for the quasilinear Schrödinger equation
\begin{equation}\label{Scr}
iz_t = -\Delta z + V(x) z -h(|z|^2)z -\Delta(l(|z|^2)) l^{\prime}(|z|^2)z \quad\mbox{ in } \R^N,
\end{equation}
where $V:\R^N\to\R$ is a given potential, $l$, $h: \R \to \R$ are real functions
and the solution $z=z(x,t)$ is a complex function in $\R^N\times\R$. 
In fact, by using the ansatz $z(x,t) = {\rm e}^{-iEt}u(x)$ in \eqref{Scr} with $E\in\R$,
the unknown strictly positive real function $u(x)$ is a solution of the 
corresponding elliptic problem
\[
-\Delta u - \Delta(l(u^2)) l^{\prime}(u^2)u +V(x) u = Eu + h(u^2) u \quad\mbox{ in } \R^N,
\]
often referred as modified Schrödinger equation, 
which matches with \eqref{problem}, but taking $p=2$, for suitable choices of function $l(s)$.
And, again, the structure of the real term $l(s)$ allows one 
to use quasilinear equation \eqref{Scr} for describing several physical phenomena 
such as the self--channeling of a high--power ultrashort laser, 
or also some problems which arise in plasma physics, fluid mechanics, 
mechanics and in the condensed matter theory (see \cite{Zhang} and references therein or also \cite{CaSp2022} for some model problems). 

On the other hand, if $A(x,t)$ is a constant, 
equation \eqref{problem} turns into the $p$--Laplacian equation 
\begin{equation}\label{noA}
-\Delta_p u +V(x)|u|^{p-2}u = g(x, u) \quad\mbox{ in } \R^N
\end{equation}
which has been widely investigated starting from the pioneer papers \cite{BW,Ra}.
Indeed, it is well known that \eqref{noA} has a variational
structure but there is a lack of compactness as
such a problem is settled in the whole space $\R^N$
and classical variational tools cannot work easily; 
thus, suitable assumptions on potential $V(x)$ are required 
(see, e.g., \cite{LZ} and references therein).

Clearly, the same difficulties still arise 
when we deal with the more general problem \eqref{problem}, but 
the presence of a coefficient which depends on the solution itself makes
the variational approach more complicated.
In fact, the ``natural'' functional associated to \eqref{problem} is 
\[
\J(u) = \frac1p \int_{\R^N} A(x, u)|\nabla u|^p dx +\frac1p \int_{\R^N} V(x) |u|^p dx -\int_{\R^N} G(x, u) dx,
\]
with $G(x,t) = \int_0^t g(x,s) ds$, 
but, even if $A(x,t)$ is a smooth strictly positive bounded function,
if $A_t(x,t) \not\equiv 0$ functional $\J$ is well defined in $W^{1,p}(\R^N)$
while it is G\^ateaux differentiable only along directions in 
$W^{1,p}(\R^N) \cap L^\infty(\R^N)$. 

Such a problem has been overcome in different ways, for example 
by introducing suitable definitions
of critical point for non--differentiable functionals 
(see, e.g., \cite{AB1,CD,CDM}), and in the whole Euclidean space $\R^N$ 
some existence results have been proved
by using nonsmooth techniques (see, e.g., \cite{AG}), 
by means of constrained minimization arguments (see \cite{LW,LWW2,PSW}),
by introducing a suitable change of variable 
but only if $A(x,t)$ has a very special form, in particular it is independent of $x$
(see, e.g., \cite{CJ,LWW,ShW} and also \cite{SC} and references therein) or by making use of an approximation scheme via $q$--Laplacian regularization (see \cite{LLW}). It is worth noting that the last mentioned technique turns out to be attractive also for studying the existence of multiple solutions of quasilinear equations of general forms to which the idea of changing variables does not apply (see \cite{LW2}) and we think it should be interesting for possible future investigations to find out the right perturbation which may allow us to apply such a regularization approach to our more general problem \eqref{problem}. 

Along with the aforementioned ideas, 
such a problem has been addressed also 
with a different approach which has been developed 
in \cite{CP1,CP2,CP3}.
More precisely, under some quite
natural conditions, we are able to prove 
that functional $\J$ is $C^1$ in the 
Banach space $X = W^{1,p}(\R^N) \cap L^\infty(\R^N)$
equipped with the intersection norm $\|\cdot\|_X$ (see Proposition \ref{C1}),
then some abstract results can be applied in this setting as long as a weaker ``compactness''
condition is provided (see Definition \ref{wCPS}). 

Recently, if $V(x) \equiv 1$ this approach has been used for proving
the existence of radial bounded solutions of \eqref{problem} 
if $A(x,t)$ is quite general but
all the involved functions are radially symmetric (see \cite{CS2020})
or they are 1--periodic with respect to $x$ (see \cite{CPSpreprint}).

Here, we get rid of both the symmetric and the periodic assumption 
and, as in the pioneer paper \cite{BF} which deals with a linear Schrödinger equation
with $A(x,t)\equiv 1$ and $p=2$
(see also \cite{AS} for the nonlinear case),
we assume that the potential $V:\R^N\to\R$ is a Lebesgue measurable function such that
\begin{equation} \label{assV}
\essinf_{x\in \R^N} V(x)>0 \qquad\mbox{ and }\qquad 
\lim_{|x|\to +\infty}\int_{B_1(x)}\frac{1}{V(y)}\ dy\ = \ 0,
\end{equation}
where $B_1(x)$ denotes the unitary sphere of $\R^N$ centered in $x$.

Such assumptions \eqref{assV} have already been used in \cite{BCS2015,BCS2016} 
for investigating the quasilinear equation \eqref{noA}
as they ensure a compactness embedding in suitable weighted Lebesgue spaces on $\R^N$
(see \cite[Theorem 3.1]{BF}).
A similar compact embedding has been stated in \cite{BPW, BW} under suitable conditions on the level sets of $V$ (for a comparison with assumption \eqref{assV} see \cite{AS}).

Anyway, in the more general setting of problem \eqref{problem} this is not enough 
and also an approximation argument is considered, i.e., 
for each $k\in\N$ we introduce the approximating quasilinear problem
\begin{equation}   \label{pbk}
\left\{
\begin{array}{ll}
-{\rm div}(A(x,u)|\nabla u|^{p-2}\nabla u) +\frac1p A_t(x,u)|\nabla u|^p = g(x,u) - V(x)|u|^{p-2} u 
 &\hbox{ in $B_k$,}\\ [10pt]
u = 0 &\hbox{ on $\partial B_k$,}
\end{array}
\right.
\end{equation}
with $B_k=\{x\in\R^N : |x| < k\}$, and the related functional $\J_{B_k}$
on the ``right''  Banach space $X_{B_k}$ (see definition \eqref{funOm}, respectively \eqref{Xlim}).
Then, if ``good'' assumptions are satisfied (see Section \ref{mainsection})
for all $k \in \N$ a weak bounded solution $u_k$ of problem \eqref{pbk} exists
and a nontrivial solution for equation \eqref{problem} is constructed as a 
suitable limit of sequence $(u_k)_k$. 

Thus, the first step is solving the given equation in bounded domains by following the ideas 
developed in \cite{CP1,CP2}, then we pass to the limit in a ``weak sense'' so to find a weak
bounded solution of \eqref{problem} which has to be nontrivial (a similar approach is used in \cite{CPSpreprint}).

We note that, since our variational principle requires that $X$ is contained in $L^{\infty}(\R^N)$ (see 
definition \eqref{Xdefn}), in some sense the ``transition'' through bounded domain is mandatory
as every ``limit point'' has to be a bounded function and the technical lemma, which allows it, 
holds only on bounded domains (see  Lemma \ref{tecnico}).

We note that our main theorem requires not only \eqref{assV} but also $V(x)$ bounded on bounded sets,
some hypotheses which describe suitable interaction properties between $A(x,t)$ and its derivative $A_t(x,t)$
and that $g(x,t)$ has a subcritical growth and satisfies an
Ambrosetti--Rabinowitz type condition and a suitable assumption  
while approaching the origin. Thus, in order to not weigh this introduction down
with too many details, we prefer to specify each
hypothesis when required and to state 
our main result at the beginning of Section \ref{mainsection} 
(see Theorem \ref{ThExist}).    

This paper is organized as follows.
 
In Section \ref{abstractsection} we introduce the abstract framework 
stating the weak Cerami--Palais--Smale condition 
and a generalized version of the classical Mountain Pass Theorem. 
In Section \ref{variationalsection} we introduce some preliminary assumptions on the functions 
$A(x,t)$, $G(x,t)$ and on the potential $V(x)$, which allow us to give a variational principle 
for problem \eqref{problem}. 
In Section \ref{mainsection} we provide some further hypotheses on the involved functions,
our main result is stated
and some geometric properties are pointed out. 
Then, in Section \ref{sectionbdd} we investigate the existence 
of weak bounded solutions of problem \eqref{pbk}
and, finally, in Section \ref{finalsection} we prove our main theorem.


\section{Abstract setting}\label{abstractsection}

Throughout this section, we assume that:
\begin{itemize}
\item $(X, \|\cdot\|_X)$ is a Banach space with dual 
$(X',\|\cdot\|_{X'})$;
\item $(W,\|\cdot\|_W)$ is a Banach space such that
$X \hookrightarrow W$ continuously, i.e. $X \subset W$ and a constant $\sigma_0 > 0$ exists
such that
\[
\|\xi\|_W \ \le \ \sigma_0\ \|\xi\|_X\qquad \hbox{for all $\xi \in X$;}
\]
\item $J : {\cal D} \subset W \to \R$ and $J \in C^1(X,\R)$ with $X \subset {\cal D}$.
\end{itemize}

In order to avoid any
ambiguity and simplify, when possible, the notation, 
from now on by $X$ we denote the space equipped with
its given norm $\|\cdot\|_X$ while, if the norm $\Vert\cdot\Vert_{W}$ is involved,
we write it explicitly.

For simplicity, taking $\beta \in \R$, we say that a sequence
$(\xi_n)_n\subset X$ is a {\sl Cerami--Palais--Smale sequence at level $\beta$},
briefly {\sl $(CPS)_\beta$--sequence}, if
\[
\lim_{n \to +\infty}J(\xi_n) = \beta\quad\mbox{and}\quad 
\lim_{n \to +\infty}\|dJ\left(\xi_n\right)\|_{X'} (1 + \|\xi_n\|_X) = 0.
\]
Moreover, $\beta$ is a {\sl Cerami--Palais--Smale level}, briefly {\sl $(CPS)$--level}, 
if there exists a $(CPS)_\beta$ -- sequence.

As $(CPS)_\beta$ -- sequences may exist which are unbounded in $\|\cdot\|_X$
but converge with respect to $\|\cdot\|_W$,
we have to weaken the classical Cerami--Palais--Smale 
condition in a suitable way according to the ideas already developed in 
previous papers (see, e.g., \cite{CP3}).  

\begin{definition} \label{wCPS}
The functional $J$ satisfies the
{\slshape weak Cerami--Palais--Smale 
condition at level $\beta$} ($\beta \in \R$), 
briefly {\sl $(wCPS)_\beta$ condition}, if for every $(CPS)_\beta$--sequence $(\xi_n)_n$,
a point $\xi \in X$ exists, such that 
\begin{description}{}{}
\item[{\sl (i)}] $\displaystyle 
\lim_{n \to+\infty} \|\xi_n - \xi\|_W = 0\quad$ (up to subsequences),
\item[{\sl (ii)}] $J(\xi) = \beta$, $\; dJ(\xi) = 0$.
\end{description}
If $J$ satisfies the $(wCPS)_\beta$ condition at each level $\beta \in I$, $I$ real interval, 
we say that $J$ satisfies the $(wCPS)$ condition in $I$.
\end{definition}

Since Definition \ref{wCPS} allows one to prove a Deformation Lemma (see \cite[Lemma 2.3]{CP3}), 
thus the following generalization of the Mountain Pass Theorem 
\cite[Theorem 2.1]{AR} can be stated (for the proof, see \cite[Theorem 1.7]{CP3}
with remarks in \cite[Theorem 2.2]{CS2020}).

\begin{theorem}[Mountain Pass Theorem]
\label{mountainpass}
Let $J\in C^1(X,\R)$ be such that $J(0) = 0$
and the $(wCPS)$ condition holds in $\R$.
Moreover, assume that two constants
$\varrho$, $\alpha^* > 0$ and a point $e \in X$ exist such that
\[
u \in X, \; \|u\|_W = \varrho\quad \then\quad J(u) \ge \alpha^*,
\]
\[
\|e\|_W > \varrho\qquad\hbox{and}\qquad J(e) < \alpha^*.
\]
Then, $J$ has a critical point $u_X \in X$ such that 
\[
J(u_X) = \inf_{\gamma \in \Gamma} \sup_{s\in [0,1]} J(\gamma(s)) \ge \alpha^*
\]
with $\Gamma = \{ \gamma \in C([0,1],X):\, \gamma(0) = 0,\; \gamma(1) = e\}$.
\end{theorem}


\section{The variational principle}  \label{variationalsection}

Let $\N = \{1,2,\dots\}$ be the set of the strictly positive integers
and, taking any $\Omega$ open subset of $\R^N$, $N\ge 2$, we denote by:
\begin{itemize}
\item $B_R(x) = \{y\in\R^N : |y-x|< R\}$ the open ball with center in 
$x\in \R^N$ and radius $R>0$;
\item $|D|$ the usual $N$--dimensional Lebesgue measure of a measurable set $D$ in $\R^N$;
\item $(L^r(\Omega),|\cdot|_{\Omega,r})$ the classical Lebes\-gue space with
norm $|u|_{\Omega,r} = \left(\int_{\Omega}|u|^r dx\right)^{1/r}$ if $1\le r<+\infty$;
\item $(L^\infty(\Omega),|\cdot|_{\Omega,\infty})$ the space of Lebesgue--measurable 
essentially bounded functions endowed with norm $\displaystyle |u|_{\Omega,\infty} = \esssup_{\Omega} |u|$;
\item $W^{1,p}(\Omega)$ and $W_0^{1,p}(\Omega)$ the classical Sobolev spaces 
both equipped with the standard norm 
$\|u\|_\Omega = (|\nabla u|_{\Omega,p}^p +|u|_{\Omega,p}^p)^{\frac1p}$
 if $1 \le p < +\infty$.
\end{itemize}
Moreover, if $V:\R^N\to\R$ is a measurable function such that
\begin{itemize}
\item[$(V_1)$] $\qquad\qquad\displaystyle\essinf_{x\in\R^N} V(x)>0$,
\end{itemize}
we denote by
\begin{itemize}
\item $(L_V^r(\Omega),|\cdot|_{\Omega,V,r})$, if $1\le r<+\infty$, the weighted Lebesgue space with
\begin{equation}   \label{LVnorm}
L_V^r(\Omega)=\left\{u\in L^r(\Omega): \int_{\Omega} V(x) |u|^r dx <+\infty\right\},\quad
|u|_{\Omega,V,r} =\left(\int_{\Omega} V(x) |u|^r dx\right)^{\frac1r};
\end{equation}
\item $W_{V}^{1,p}(\Omega)$ and $W_{0,V}^{1,p}(\Omega)$, if $1\le p<+\infty$, 
the weighted Sobolev spaces
\[
\begin{split}
W_{V}^{1,p}(\Omega) = &\left\{ u\in W^{1,p}(\Omega): \ \int_{\Omega} V(x) |u|^p dx <+\infty\right\},\\
W_{0,V}^{1,p}(\Omega) = &\left\{ u\in W_0^{1,p}(\Omega): \ \int_{\Omega} V(x) |u|^p dx <+\infty\right\}
\end{split}
\]
endowed with the norm
\begin{equation}   \label{weightnorm}
\|u\|_{\Omega,V} = (|\nabla u|_{\Omega,p}^p +|u|_{\Omega,V,p}^p)^{\frac1p}.
\end{equation}
\end{itemize}

For simplicity, we put $B_R = B_R(0)$ for the open ball with center in the origin 
and radius $R>0$ and, if $\Omega = \R^N$, we avoid to write the set in the norms, i.e., 
we put
\begin{itemize}
\item $|\cdot|_{r} = |\cdot|_{\R^N,r}$ for the norm in $L^r(\R^N)$, for all $1 \le r \le +\infty$;
\item $|\cdot|_{V,r} = |\cdot|_{\R^N,V,r}$ for the norm in $L_V^r(\R^N)$, for all $1 \le r < +\infty$;
\item $\|\cdot\| = \|\cdot\|_{\R^N}$ for the norm in $W^{1,p}(\R^N) = W^{1,p}_0(\R^N)$;
\item $\|\cdot\|_{V} = \|\cdot\|_{\R^N,V}$ for the norm in 
$W_V^{1,p}(\R^N) = W_{0,V}^{1,p}(\R^N)$.
\end{itemize}  

From now on, let $V:\R^N\to\R$ be a measurable function which satisfies not only condition $(V_1)$
but also
\begin{itemize}
\item[$(V_2)$] $\qquad\qquad\displaystyle\lim_{|x|\to +\infty}\int_{B_1(x)}\frac{1}{V(y)}\ dy\ =\ 0$.
\end{itemize}

\begin{remark} \label{Remb}
If potential $V(x)$ satisfies assumption $(V_1)$, then the following
continuous embeddings hold:
\begin{equation}   \label{LVem}
L_V^r(\R^N)\hookrightarrow L^r(\R^N) \quad \hbox{for all $1\le r<+\infty$,}
\end{equation}
\begin{equation}  \label{WVem}
W_V^{1, p}(\R^N)\hookrightarrow W^{1, p}(\R^N)
\quad \hbox{for all $1\le p<+\infty$.}
\end{equation}
\end{remark}

From Remark \ref{Remb} and Sobolev Embedding Theorems, we deduce the following result 
(for the compact embeddings, see \cite[Theorem 3.1]{BF}). 

\begin{theorem}\label{embed}
Let $V:\R^N\to\R$ be a Lebesgue measurable function satisfying assumption $(V_1)$. 
Then, the following continuous embeddings hold:
\begin{itemize}
\item if $p<N$ then
\begin{equation}  \label{cont1}
W_V^{1,p}(\R^N) \hookrightarrow L^r(\R^N) \quad\mbox{ for any }\; p\le r\le \frac{Np}{N-p};
\end{equation}
\item if $p=N$ then
\begin{equation}  \label{cont2}
W_V^{1, p}(\R^N)\hookrightarrow L^r(\R^N) \quad\mbox{ for any }\; p\le r<+\infty;
\end{equation}
\item if $p>N$ then
\begin{equation}  \label{cont3}
W_V^{1, p}(\R^N)\hookrightarrow L^{r}(\R^N)  \quad\mbox{ for any }\; p\le r\le+\infty.
\end{equation}
\end{itemize}
Furthermore, if assumption $(V_2)$ also occurs, the compact embedding
\begin{equation}    \label{comp}
W_V^{1, p}(\R^N)\hookrightarrow\hookrightarrow L^r(\R^N) \quad\mbox{ for any } p\le r <p^*
\end{equation}
holds, with
\[
p^*=\begin{cases}
\frac{Np}{N-p} &\hbox{ if } p<N,\\
+\infty &\hbox{ if } p\ge N.
\end{cases}
\]
\end{theorem}

From Theorem \ref{embed} it follows that, 
for any $r \ge p$ so that \eqref{cont1}, respectively \eqref{cont2}
or \eqref{cont3}, holds then 
a constant $\tau_r>0$ exists such that
\begin{equation}   \label{sob}
|u|_r\le\tau_r \|u\|_{V} \quad\mbox{ for all } u\in W_V^{1, p}(\R^N).
\end{equation}

On the other hand, taking $\Omega$ open bounded domain in $\R^N$ 
and $p<N$, a classical embedding theorem implies that
a constant $\sigma_* > 0$, independent of $\Omega$ and 
depending only on $p$ and $N$, exists such that
\begin{equation}\label{stima*}
|u|_{\Omega,p^*}\ \le\ \sigma_* \|u\|_{\Omega} \quad \hbox{for all $u \in W_0^{1,p}(\Omega)$.}
\end{equation}

From now on, we assume that potential $V(x)$ is a measurable function 
which satisfies condition $(V_1)$ and we set
\begin{equation}  \label{Xdefn}
X:= W_V^{1, p}(\R^N)\cap L^{\infty}(\R^N)\quad \hbox{and}\quad 
\|u\|_X =\|u\|_V +|u|_{\infty}\quad\mbox{ for any } u\in X.
\end{equation}

In the following, we assume $p\le N$ (otherwise, embedding \eqref{cont3} implies that $X=W_V^{1, p}(\R^N)$ 
and all the arguments and proofs can be simplified).

\begin{lemma}   \label{lemmaX}
Taking $p \le r < +\infty$, we have that
\begin{equation} \label{Lwin}
X \hookrightarrow L_V^{r}(\R^N)  \qquad \hbox{with}\quad
|u|_{V,r}\le \|u\|_X \quad\mbox{ for all } u\in X.
\end{equation}
\end{lemma}

\begin{proof}
Taking any $u\in X$, definitions \eqref{weightnorm} and \eqref{Xdefn} imply that
\[
\int_{\R^N} V(x) |u|^r dx \le |u|_{\infty}^{r-p}\int_{\R^N} V(x) |u|^p dx
\le |u|_{\infty}^{r-p} \|u\|_V^ p\le \|u\|_X^r
\]
which gives the proof.
\end{proof}

\begin{remark}  \label{remX}
From \eqref{LVem}, definition \eqref{Xdefn} and Lemma \ref{lemmaX} it follows that
\[
X\hookrightarrow L^r(\R^N) \quad 
\mbox{ for any $p \le r \le +\infty$.}
\]
\end{remark}

\begin{lemma}\label{immergo2}
If $(u_n)_n \subset X$, $u \in X$ and $M > 0$ are such that
\begin{equation}    \label{conW}
\|u_n -u\|_{V}\to 0 \quad\mbox{ as } n\to+\infty
\end{equation}
and
\begin{equation} \label{<M}
|u_n|_{\infty}\le M \quad\mbox{ for all } n\in\N,
\end{equation}
then $u_n \to u$ strongly in $L_V^r(\R^N)$ for any $p \le r < +\infty$. 
\end{lemma}

\begin{proof}
From \eqref{<M} we have that
\[
\int_{\R^N} V(x) |u_n-u|^r dx \le|u_n-u|_{\infty}^{r-p}\int_{\R^N}V(x) |u_n -u|^p dx \le (M+|u|_{\infty})^{r-p}\|u_n -u\|_{V}^p,
\]
which, together with $(V_1)$, \eqref{LVnorm} and \eqref{conW}, implies the desired result.
\end{proof}

We proceed recalling the following definition.

\begin{definition}
A function $h:\R^N\times\R\to\R$ is a $\mathcal{C}^{k}$--Carath\'eodory function, 
$k\in\N\cup\lbrace 0\rbrace$, if
\begin{itemize}
\item $h(\cdot,t) : x \in \R^N \mapsto h(x,t) \in \R$ is measurable for all $t \in \R$,
\item $h(x,\cdot) : t \in \R \mapsto h(x,t) \in \R$ is $\mathcal{C}^k$ for a.e. $x \in \R^N$.
\end{itemize}
\end{definition}

Let $A:(x,t) \in \R^N\times\R \mapsto A(x,t) \in \R$ be a given function such that the following conditions hold:
\begin{itemize}
\item[$(h_0)$]
$A(x,t)$ is a $\mathcal{C}^1$--Carath\'eodory function with $A_t(x,t)=\frac{\partial}{\partial t}A(x, t)$;
\item[$(h_1)$] for any $\rho > 0$ we have that
\[
\sup_{\vert t\vert\leq \rho} \vert A(\cdot,t)\vert \in L^{\infty}(\R^N), \ 
\quad \sup_{\vert t\vert\leq \rho} \vert A_t(\cdot,t)\vert \in L^{\infty}(\R^N).
\]
\end{itemize}

Furthermore, we assume that $g:\R^N\times\R\to\R$ exists such that:
\begin{itemize}
\item[$(g_0)$] $g(x,t)$ is a $\mathcal{C}^0$--Carath\'eodory function;
\item[$(g_1)$] $a_1, a_2 >0$ and $q\ge p$ exist such that
\[
|g(x, t)|\le a_1 |t|^{p-1} + a_2 |t|^{q-1}\quad \mbox{ a.e. in } \R^N, \mbox{ for all } t\in\R.
\] 
\end{itemize}

\begin{remark}   
Assumptions $(g_0)$ and $(g_1)$ imply that 
\begin{equation}  \label{Gdefn}
G:(x, t)\in\R^N\times\R\mapsto\int_0^t g(x, s) ds\in\R
\end{equation}
is a well defined $\mathcal{C}^1$--Carath\'eodory function and
\begin{equation}\label{Gle}
|G(x, t)|\le \frac{a_1}{p}|t|^p +\frac{a_2}{q}|t|^q \quad \mbox{ a.e. in } \R^N, \mbox{ for all } t\in\R.
\end{equation}
In particular, $(g_1)$ and \eqref{Gdefn} imply that
\begin{equation} \label{G0=0}
G(x, 0) =g(x, 0)=0\quad \mbox{ for a.e. } x\in\R^N.
\end{equation}
\end{remark}

Taking any $u\in X$, from assumption $(h_1)$ and definition \eqref{Xdefn} 
it follows that $A(\cdot, u)|\nabla u(\cdot)|^p\in L^1(\R^N)$, 
while hypotheses $(g_0)$, $(g_1)$ and Remark \ref{remX} provide 
that $G(\cdot, u)\in L^1(\R^N)$. 
Then, assumption $(V_1)$ implies that functional
\begin{equation}  \label{funct}
\J(u) = \frac1p \int_{\R^N} A(x, u)|\nabla u|^p dx +\frac1p \int_{\R^N} V(x) |u|^p dx -\int_{\R^N} G(x, u) dx
\end{equation}
is well defined for all $u\in X$. 
Moreover, taking any $u$, $v\in X$, the same assumptions imply
that $A_t(\cdot,u)v |\nabla u(\cdot)|^p\in L^1(\R^N)$
and $g(\cdot, u) v\in L^1(\R^N)$, then the G\^ateaux differential 
of functional $\J$ in $u$ along the direction $v$ is well defined and is given by
\begin{equation}   \label{diff}
\begin{split}
\langle d\J(u), v\rangle &=\int_{\R^N} A(x, u)|\nabla u|^{p-2} \nabla u\cdot\nabla v dx 
+\frac1p \int_{\R^N} A_t(x, u) v|\nabla u|^p dx\\
&\quad +\int_{\R^N} V(x) |u|^{p-2} u v dx -\int_{\R^N} g(x, u) v dx.
\end{split}
\end{equation}

As useful in the following, we recall same classical inequalities (for the proof, see, e.g., \cite{GM}).

\begin{lemma}      \label{lemmavett}
A constant $C_0 > 0$ exists such that for any $\xi, \eta\in\R^N, N\geq 1$, it results
\begin{align}   \label{>2}
|| \xi|^{r-2} \xi -| \eta|^{r-2} \eta| &\leq
C_0 |\xi -\eta|\left(| \xi\vert +\vert \eta|\right)^{r-2} &&\hbox{ if } r > 2,\\  
\label{=2}
|| \xi|^{r-2} \xi -|\eta|^{r-2} \eta|
&\leq C_0 | \xi -\eta|^{r-1} &&\hbox{ if } 1< r\leq 2.
\end{align}
\end{lemma}

By reasoning as in \cite[Lemma 3.4]{CS2020},
from Lemma \ref{lemmavett} we deduce the following result.

\begin{lemma}\label{lemmaOm}
Taking $p > 1$, a constant $C_1 = C_1(p) > 0$ exists such that for any open domain $\Omega\subset\R^N$ it results
\[
\int_{\Omega}||\nabla w|^p-|\nabla z|^p|dx\le C_1 \|w-z\|_{\Omega}\left(\|w\|_{\Omega}^{p-1} +\|z\|_{\Omega}^{p-1}\right) 
\quad\forall w, z\in W_0^{1, p}(\Omega).
\]
\end{lemma}

Now, we are ready to state a ``good'' variational principle.

\begin{proposition} \label{C1}
Let $p>1$ and suppose that hypotheses $(V_1)$, $(h_0)$--$(h_1)$ and $(g_0)$--$(g_1)$ hold. 
If $(u_n)_n\subset X$ and $u\in X$ are such that
\begin{equation}   \label{reg3}
u_n\to u \quad\mbox{ a.e. in } \R^N
\end{equation}
and \eqref{conW}, \eqref{<M} hold for a constant $M>0$, then
\[
\J(u_n)\to \J(u) \quad\mbox{ and }\quad \|d\J(u_n) -d\J(u)\|_{X^{\prime}}\to 0\quad\mbox{ as } n\to +\infty.
\]
Hence, $\J$ is a $\mathcal{C}^1$ functional in $X$ with Fr\'echet differential  
defined as in \eqref{diff}.
\end{proposition}

\begin{proof}
For the sake of convenience, we set $\J=\J_1+\J_2 - \J_3$, where
\[
\begin{split}
&\J_1:u\in X\mapsto \J_1(u) = \frac{1}{p}\ \int_{\R^N} A(x, u) |\nabla u|^p dx \in\R,\\  
&\J_2:u\in X\mapsto \J_2(u) = \frac1p \int_{\R^N} V(x) |u|^p dx\in\R,\\  
&\J_3:u\in X\mapsto\J_3(u) = \int_{\R^N} G(x, u) dx \in\R,
\end{split}
\]
with related G\^ateaux differentials
\[
\begin{split}
&\langle d\J_1(u), v\rangle =\int_{\R^N} A(x, u)|\nabla u|^{p-2} \nabla u\cdot\nabla v\ dx 
+\frac1p \int_{\R^N} A_t(x, u) v|\nabla u|^p dx,\\
&\langle d\J_2(u), v\rangle = \int_{\R^N} V(x) |u|^{p-2} u v\ dx,\\
&\langle d\J_3(u), v\rangle = \int_{\R^N} g(x, u) v\ dx,
\end{split}
\]
for any $u$, $v \in X$.\\
Now, let $(u_n)_n\subset X$, $u\in X$ and $M>0$ be such that \eqref{conW}, \eqref{<M} and \eqref{reg3} hold. \\
Firstly, we note that \eqref{WVem} and \eqref{conW} imply
$\|u_n - u\| \to 0$, then from the proof of \cite[Proposition 3.6]{CS2020}
it follows that
\[
\begin{split}
&\J_1(u_n)\to\J_1(u)\quad\mbox{ and }\quad \|d\J_1(u_n)-d\J_1(u)\|_{X'}\to 0,\\
&\J_3(u_n)\to\J_3(u)\quad\mbox{ and }\quad \|d\J_3(u_n)-d\J_3(u)\|_{X'}\to 0.
\end{split}
\]
Moreover, from definitions \eqref{LVnorm} and \eqref{weightnorm}, 
limit \eqref{conW} implies also that
\[
|\J_2(u_n) - \J_2(u)| = \frac1p ||u_n|_{V,p}^p -|u|_{V,p}^p| \to 0
\quad \hbox{as $n\to +\infty$}.
\]
At last, taking $v\in X$ such that $\|v\|_X=1$, so 
\begin{equation}   \label{<1}
|v|_{\infty}\le 1, \quad |v|_{V,p} \le 1,
\end{equation}
from $(V_1)$ and by definition 
we have that
\begin{equation}  \label{div1}
|\langle d\J_2(u_n) - d\J_2(u), v\rangle |\ \le\ 
\int_{\R^N} V(x) ||u_n|^{p-2} u_n -|u|^{p-2} u| |v| dx.
\end{equation}
Thus, if $1 < p \le 2$, from \eqref{=2}, 
Hölder inequality with $V(x) =V^{\frac1p}(x) V^{\frac{p-1}{p}}(x)$
and \eqref{<1} we have that
\begin{equation} \label{V2}
\begin{split}
&\int_{\R^N} V(x) ||u_n|^{p-2} u_n -|u|^{p-2} u| |v| dx \le 
C_0 \int_{\R^N} V(x) |u_n -u|^{p-1} |v| dx\\
&\qquad\le C_0 |u_n -u|_{V,p}^{p-1} |v|_{V,p} \le C_0 |u_n -u|_{V,p}^{p-1}.
\end{split}
\end{equation}
On the other hand, if $p>2$, again from Hölder inequality 
with $V(x) =V^{\frac1p}(x) V^{\frac{p-1}{p}}(x)$,
and \eqref{>2}, \eqref{<1} and direct computations we have that
\begin{equation}   \label{V1}
\begin{split}
&\int_{\R^N} V(x) ||u_n|^{p-2} u_n -|u|^{p-2} u| |v| dx \le 
\left(\int_{\R^N} V(x) ||u_n|^{p-2} u_n -|u|^{p-2} u|^{\frac{p}{p-1}} dx\right)^{\frac{p-1}{p}} |v|_{V,p}\\
&\qquad 
\le C_0 \left(\int_{\R^N} V(x) |u_n -u|^{\frac{p}{p-1}}(|u_n|+|u|)^{\frac{p(p-2)}{p-1}} dx\right)^{\frac{p-1}{p}},
\end{split}
\end{equation}
where once again from Hölder inequality but
with $V(x) =V^{\frac1{p-1}}(x) V^{\frac{p-2}{p-1}}(x)$ it results
\[
\left(\int_{\R^N} V(x) |u_n -u|^{\frac{p}{p-1}}(|u_n|+|u|)^{\frac{p(p-2)}{p-1}} dx\right)^{\frac{p-1}{p}}
\le |u_n -u|_{V,p} \left(\int_{\R^N} V(x) (|u_n|+|u|)^{p} dx\right)^{\frac{p-2}{p}}.
\]
Hence, from \eqref{V1}, direct computations imply that
\begin{equation}  \label{V3}
\int_{\R^N} V(x) ||u_n|^{p-2} u_n -|u|^{p-2} u| |v| dx \le 
C^*_0 |u_n -u|_{V,p} \big(|u_n|_{V,p}^{p-2} + |u|_{V,p}^{p-2}\big)
\end{equation}
for a suitable constant $C^*_0 >0$.\\
Thus, summing up, from \eqref{div1} and \eqref{V2}, respectively \eqref{V3}, it follows that
\[
|\langle d\J_2(u_n) -d\J_2(u), v\rangle|\to 0\quad\mbox{ uniformly with respect to } v\in X,\; \|v\|_X=1,
\]
and the desired result holds.
\end{proof}


\section{The main theorem and some first properties} \label{mainsection}

From now on, besides hypotheses $(V_1)$, $(h_0)$--$(h_1)$, $(g_0)$--$(g_1)$, which imply that
$\J$ as in \eqref{funct} is a $\mathcal{C}^1$ functional on $X$ as in \eqref{Xdefn} 
(see Proposition \ref{C1}), we consider also 
assumption $(V_2)$ and the following conditions on functions $A(x,t)$, $g(x,t)$ and $V(x)$:
\begin{itemize}
\item[$(h_2)$] a constant $\alpha_0>0$ exists such that
\[
A(x, t)\ge\alpha_0 \quad\mbox{ a.e. in } \R^N, \mbox{ for all } t\in\R;
\]
\item[$(h_3)$] some constants $\mu>p$ and $\alpha_1>0$ exist so that
\[
(\mu-p)A(x, t)-A_t(x, t)t\ge\alpha_1 A(x, t) \quad\mbox{ a.e. in } \R^N, \mbox{ for all } t\in\R;
\]
\item[$(h_4)$] a constant $\alpha_2>0$ exists such that
\[
pA(x, t) + A_t(x, t)t\ge\alpha_2 A(x, t) \quad\mbox{ a.e. in } \R^N \mbox{ for all } t \in \R;
\]
\item[$(g_2)$] we have that
\[
\lim_{t\to 0} \frac{g(x, t)}{|t|^{p-2} t} =\bar{\alpha}\
<\ \frac{\alpha_0}{\tau_p^p} \quad\mbox{ uniformly a.e. in } \R^N,
\]
where $\alpha_0$ is as in assumption $(h_2)$ and $\tau_p$ is as in \eqref{sob} with $r=p$;
\item[$(g_3)$] having $\mu$ as in hypothesis $(h_3)$, then
\[
0<\mu G(x, t)\le g(x, t)t \quad\mbox{ a.e. in } \R^N, \mbox{ for all } t\in\R\setminus\{0\};
\]
\item[$(V_3)$] for any $\varrho >0$, a constant $C_{\varrho}>0$ exists such that
\[
\esssup_{|x|\le\varrho} V(x)\le C_{\varrho}.
\]
\end{itemize}

\begin{remark} \label{RemV1}
If $(h_2)$ holds, then, without loss of generality, we can assume 
$\alpha_0 \le 1$. Moreover, taking $t=0$ in $(h_3)$, we have
also $\mu -p\ge\alpha_1$.
\end{remark}

\begin{remark}
By means of \eqref{Gdefn}, assumptions $(g_0)$ and $(g_2)$
imply that the existence of
\[
\lim_{t\to 0}\frac{G(x, t)}{|t|^p} =\frac{\bar{\alpha}}{p} \quad\mbox{ uniformly a.e. in } \R^N
\]
is provided, too. In particular, from hypotheses $(g_1)$--$(g_2)$ and direct computations 
it follows that for any $\varepsilon >0$ a constant $c_{\varepsilon} > 0$ exists so that not only
\[
|g(x, t)|\le (\bar{\alpha} +\varepsilon)|t|^{p-1} +c_{\varepsilon}|t|^{q-1}
\quad\mbox{ a.e. in } \R^N, \mbox{ for all } t\in\R,
\]
but also
\begin{equation}   \label{Glimcom}
|G(x,t)|\le \frac{\bar{\alpha}+\varepsilon}{p}|t|^p +\frac{c_{\varepsilon}}{q} |t|^q 
\quad\qquad\mbox{  a.e. in } \R^N, \mbox{ for all } t\in\R,
\end{equation}
if assumption $(g_0)$ holds, too.
\end{remark}

\begin{remark}  
We note that \eqref{Gdefn}, assumptions $(g_0)$--$(g_1)$, $(g_3)$ and straightforward computations 
imply that for any $\varepsilon>0$ a function $\eta_{\varepsilon}\in L^{\infty}(\R^N)$
exists so that $\eta_{\varepsilon}>0$ for a.e. $x\in\R^N$ and
\begin{equation}  \label{Ggeq}
G(x,t)\ge\eta_{\varepsilon}(x) |t|^{\mu} \quad\mbox{ for a.e. } x\in\R^N \hbox{ if } |t|\ge \varepsilon.
\end{equation}
Hence, from \eqref{Gle} and \eqref{Ggeq} it follows that
\begin{equation}\label{p<q}
p<\mu \le q.
\end{equation}
\end{remark}

Now, we are ready to state our main result.

\begin{theorem}\label{ThExist}
Under assumptions $(V_1)$--$(V_3)$, $(h_0)$--$(h_4)$ and $(g_0)$--$(g_3)$,
with the growth exponent $q$ in $(g_1)$ such that 
\begin{equation} \label{subc0}
q <p^*,
\end{equation}
then problem \eqref{problem} admits at least one weak nontrivial bounded solution.
\end{theorem}

\begin{remark}
In our set of hypotheses, summing up \eqref{p<q} and \eqref{subc0} 
we have that
\[
1<p<\mu\leq q <p^*,
\]
with $\mu$ as in $(h_3)$ and $(g_3)$.
\end{remark}

As useful in the following, firstly we give a convergence result.

\begin{proposition}\label{convRN}
Suppose that hypotheses $(V_1)$--$(V_2)$, $(h_0)$--$(h_3)$,
$(g_0)$--$(g_1)$ and $(g_3)$ hold. 
Thus, taking any $\beta\in\R$, we have that any $(CPS)_{\beta}$--sequence $(u_n)_n\subset X$ 
is bounded in $W^{1, p}_V(\R^N)$. 
Furthermore, $u\in W^{1, p}_V(\R^N)$ exists such that, up to subsequences,
as $n\to +\infty$ the following limits hold:
\begin{align}   \label{conv1}
&u_n\rightharpoonup u \mbox{ weakly in } W_V^{1, p}(\R^N),\\    
\label{conv2}
&u_n\to u \mbox{ strongly in } L^r(\R^N) \mbox{ for each } r\in[p, p^*[ ,\\     
\label{conv3}
&u_n\to u \mbox{ a.e. in } \R^N.
\end{align}
\end{proposition}

\begin{proof}
Taking $\beta\in\R$, let $(u_n)_n\subset X$ be a $(CPS)_{\beta}$--sequence of $\J$, i.e.,
\begin{equation}  \label{beta0}
\J(u_n)\to\beta\quad\mbox{ and }\quad \|d\J(u_n)\|_{X^{\prime}}(1+\|u_n\|_X)\to 0 \qquad\mbox{ as } n\to +\infty.
\end{equation}
Thus, from \eqref{funct}, \eqref{diff}, \eqref{beta0}, 
assumptions $(h_2)$--$(h_3)$, $(g_3)$ and also \eqref{weightnorm},
direct computations give
\[
\begin{split}
\mu\beta +\varepsilon_n = &\mu\J(u_n) -\langle d\J(u_n), u_n\rangle\\
\ge &\frac{\alpha_0\alpha_1}{p}\int_{\R^N}|\nabla u_n|^p dx+\frac{\mu-p}{p}\int_{\R^N} V(x)|u_n|^p dx
\ge c_1\|u_n\|_V^p,
\end{split}
\]
for a suitable constant $c_1 > 0$. 
So, $(u_n)_n$ is bounded in $W^{1, p}_V(\R^N)$ and \eqref{conv1}--\eqref{conv3} follows from \eqref{comp}.
\end{proof}

Now, we prove that the geometrical assumptions needed 
to apply Theorem \ref{mountainpass} hold. To this aim, 
we start giving the following lemma.

\begin{lemma}
Under assumptions $(h_0)$ and $(h_2)$--$(h_3)$, we have that
\begin{equation}\label{Axst}
A(x,st) \le s^{\mu-p-\alpha_1}A(x,t) \quad\mbox{ a.e. in } \R^N, \; \mbox{ for all $t\in\R$
and $s\ge 1$}.
\end{equation}
\end{lemma}

\begin{proof}
Taking $t\in\R$, for a.e. $x\in\R^N$ and for all $s>0$, assumptions $(h_0)$ and $(h_3)$ imply that
\[
\frac{d}{ds} A(x, st) = A_t(x, st) t\le\ \frac{\mu - p-\alpha_1}{s} \ A(x,st).
\]
Hence, from hypothesis $(h_2)$ we obtain that 
\[
\frac{\frac{d}{ds} A(x, st)}{ A(x, st)}\ \le\ \frac{\mu - p-\alpha_1}{s},
\]
where Remark \ref{RemV1} ensures that $\mu-p-\alpha_1\ge 0$.
Thus, if $s \ge 1$ by integrating we obtain the desired result.
\end{proof}

\begin{proposition} \label{gerho}
Under assumptions $(V_1)$, $(h_0)$--$(h_2)$, $(g_0)$--$(g_2)$, if $p<q<p^*$ 
then two positive constants $\varrho$, $\alpha^* > 0$ exist so that
\begin{equation}\label{geoN}
u\in X,\; \|u\|_{V}=\varrho \quad\implies\quad \J(u)\ge\alpha^*.
\end{equation}
\end{proposition}

\begin{proof}
Let $u\in X$ and, from $(g_2)$, take $\varepsilon >0$ such that 
$\bar{\alpha} + \varepsilon <\frac{\alpha_0}{\tau_p^p}$,
i.e.,
\begin{equation}\label{geoN1}
\alpha_0 - (\bar{\alpha}+\varepsilon) \tau_p^p > 0.
\end{equation}
Then, from $(h_2)$ with $\alpha_0 \le 1$ (see Remark \ref{RemV1}), 
\eqref{Glimcom} and \eqref{sob}, definitions \eqref{weightnorm} and \eqref{funct}
imply that
\[
\begin{split}
\J(u)&\ge\frac{\alpha_0}{p}\left(\int_{\R^N}|\nabla u|^p dx + \int_{\R^N} V(x)|u|^p dx\right) 
-\frac{\bar{\alpha}+\varepsilon}{p}\int_{\R}|u|^p dx -\frac{c_{\varepsilon}}{\mu}\int_{\R^N} |u|^q dx\\
&\ge\ \frac{1}{p}\ \big(\alpha_0 - (\bar{\alpha}+\varepsilon) \tau_p^p\big) \|u\|_V^p 
- \frac{c_{\varepsilon}}{\mu} \tau_q^q \|u\|_V^q.
\end{split}
\]
Then, \eqref{geoN1} and $p<q$
allow us to find two positive constants $\varrho$ and $\alpha^*$ 
so that \eqref{geoN} holds.
\end{proof}

\begin{proposition} \label{menoinf}
Assume that $(V_1)$, $(h_0)$, $(h_2)$--$(h_3)$, $(g_0)$--$(g_1)$ and $(g_3)$ hold.
Then, fixing $\bar{u}\in X\setminus\{0\}$,
it follows that
\[
\J(s\bar{u})\to -\infty \quad\mbox{ as } s\to +\infty.
\]
\end{proposition}
\begin{proof}
Since $\bar{u}\neq 0$, there exists $\varepsilon>0$ such that
\[
|\Omega^{\bar{u}}_{\varepsilon}|>0  \qquad \hbox{with }\quad \Omega^{\bar{u}}_{\varepsilon}= \{ x\in\R^N :\ |\bar{u}(x)|>\varepsilon\}.
\]
From \eqref{funct}, \eqref{Axst}, inequality \eqref{Ggeq} and direct computations, for any $s\ge 1$ we have that
\[
\begin{split}
\J(s\bar{u})&=\frac{s^p}{p}\int_{\R^N} A(x, s\bar{u})|\nabla \bar{u}|^p dx 
+\frac{s^p}{p}\int_{\R^N} V(x) |\bar{u}|^p dx -\int_{\R^N} G(x, s\bar{u}) dx\\
&\le \frac{s^{\mu-\alpha_1}}{p}\int_{\R^N} A(x, \bar{u})|\nabla \bar{u}|^p dx 
+\frac{s^p}{p} \int_{\R^N} V(x)|\bar{u}|^p dx
- s^{\mu}\int_{\Omega^{\bar{u}}_{\varepsilon}} \eta_1(x) |\bar{u}|^{\mu} dx.
\end{split}
\]
Thus, being $\int_{\Omega^{\bar{u}}_{\varepsilon}} \eta_1(x) |\bar{u}|^{\mu} dx > 0$, 
inequality \eqref{p<q}, together with assumption $\alpha_1>0$, gives the desired result.
\end{proof}


\section{An existence result on bounded domains}\label{sectionbdd}

From now on, let $\Omega$ denote an open bounded domain in $\R^N$. Thus, we define
\begin{equation}\label{Xlim}
X_{\Omega} = W_{0,V}^{1, p}(\Omega)\cap L^{\infty}(\Omega)
\end{equation}
endowed with the norm
\begin{equation}\label{normB}
\|u\|_{X_{\Omega}} =\|u\|_{\Omega,V} +|u|_{\Omega,\infty} \quad\mbox{ for any } u\in X_{\Omega}
\end{equation}
and with dual space $X_{\Omega}^{\prime}$. \\
Actually, since any function $u \in  X_{\Omega}$ can be trivially extended 
to a function $\tilde{u}\in X$ just assuming $\tilde{u}(x) = 0$ for all $x \in \R^N \setminus \Omega$, 
then
\begin{equation}\label{Xequ}
\|\tilde{u}\| = \|u\|_{\Omega}, \quad
\|\tilde{u}\|_V = \|u\|_{\Omega,V}, 
\quad  |\tilde{u}|_{\infty} = |u|_{\Omega,\infty},\quad
\|\tilde{u}\|_X = \|u\|_{X_\Omega}.
\end{equation}

\begin{remark} \label{Rmbdd}
As $\Omega$ is a bounded domain, 
not only $\|u\|_{\Omega}$ and $|\nabla u|_{\Omega,p}$ are equivalent norms  
but also, if assumption $(V_3)$ holds, a constant $c_{\Omega} \ge 1$ exists such that
\[
\|u\|_{\Omega,V}^p = \int_{\Omega} |\nabla u|^p dx + \int_{\Omega} V(x) |u|^p dx
\le \int_{\Omega} |\nabla u|^p dx +c_{\Omega} \int_{\Omega}|u|^p dx\le c_\Omega \|u\|_{\Omega}^p,
\]
which, together with \eqref{WVem} and \eqref{Xequ}, implies that 
the norms $\|\cdot\|_{\Omega,V}$ and $\|\cdot\|_{\Omega}$ are equivalent, too. 
\end{remark}

From \eqref{G0=0} and \eqref{funct} it follows that the restriction of the functional 
$\J$ to $X_\Omega$, namely 
\[
\J_\Omega = \J|_{X_{\Omega}},
\]
is such that
\begin{equation}  \label{funOm}
\J_{\Omega}(u) = \frac1p \int_{\Omega} A(x, u)|\nabla u|^p dx 
+ \frac1p \int_{\Omega} V(x) |u|^p dx -\int_{\Omega} G(x, u) dx, \quad u\in X_{\Omega}.
\end{equation}

We note that, setting
\begin{equation}\label{gtilde}
\tilde{g}(x,t) = g(x,t) - V(x) |t|^{p-2} t
\quad\mbox{ for a.e. } x \in\R^N, \; \mbox{ all $t\in\R$,}
\end{equation}
from $(g_0)$ we have that $\tilde{g}$ is a $\mathcal{C}^0$--Caratheodory function
such that
\begin{equation}\label{gtilde1}
\tilde{G}(x,t) = \int_0^t \tilde{g}(x,s) ds = 
G(x,t) - \frac1p V(x) |t|^{p} 
\quad\mbox{ for a.e. } x \in\R^N, \; \mbox{ all $t\in\R$.}
\end{equation}
Then, by means of assumptions $(g_1)$, $(V_1)$ and $(V_3)$, we have that
\begin{equation}\label{gtildeup}
|\tilde{g}(x, t)|\le (a_1 + |V|_{\Omega,\infty}) |t|^{p-1} +a_2 |t|^{q-1}
\quad\mbox{ for a.e. } x \in\R^N, \; \mbox{ all $t\in\R$,}
\end{equation}
and from \eqref{funOm} and \eqref{gtilde1}
we have that 
\begin{equation}  \label{funOm1}
\J_{\Omega}(u) = \frac1p \int_{\Omega} A(x, u)|\nabla u|^p dx 
- \int_{\Omega} \tilde{G}(x,u) dx, \quad u\in X_{\Omega}.
\end{equation}
Hence, arguing as in \cite[Proposition 3.1]{CP2}, 
it follows that $\J_{\Omega}:X_{\Omega}\to\R$ is a 
$\mathcal{C}^1$ functional such that, for any $u$, $v\in X_{\Omega}$, 
its  Fréchet differential in $u$ along the direction $v$ is given by
\begin{equation}\label{diffOm}
\begin{split}
\langle d\J_{\Omega}(u), v\rangle =\ &\int_{\Omega} A(x, u)|\nabla u|^{p-2} \nabla u\cdot\nabla v dx 
+\frac1p \int_{\Omega} A_t(x, u) v|\nabla u|^p dx\\
& +\int_{\Omega} V(x) |u|^{p-2} u v dx -\int_{\Omega} g(x, u) v dx.
\end{split}
\end{equation}

In order to apply Theorem \ref{mountainpass} to functional
$\J_\Omega$ as in \eqref{funOm}, 
we start proving that it satisfies the $(wCPS)$ condition in $\R$ 
(see Definition \ref{wCPS}).

\begin{proposition} \label{wCPSbdd}
Suppose that hypotheses $(V_1)$, $(V_3)$, $(h_0)$--$(h_4)$, $(g_0)$--$(g_1)$
and $(g_3)$ hold. Then, if also \eqref{subc0} is verified,
functional $\J_{\Omega}$ satisfies the $(wCPS)$ condition in $\R$.
\end{proposition}

\begin{proof}
Firstly, we note that from Remark \ref{Rmbdd}
the norm in \eqref{normB} can be replaced with the equivalent one, 
still denoted by $\|\cdot\|_{X_{\Omega}}$, given by
\[
\|u\|_{X_{\Omega}} = |\nabla u|_{\Omega,p} + |u|_{\Omega,\infty} 
\quad\mbox{ for any } u\in X_{\Omega}.
\]
Then, being $\mu > p$, from $(g_3)$, 
\eqref{gtilde}--\eqref{gtilde1} and direct computations 
we obtain that
\[
\mu \tilde{G}(x,t) \le \tilde{g}(x,t) t
\quad\mbox{ for a.e. } x \in\R^N, \; \mbox{ all $t\in\R$.}
\]
Thus, from conditions $(h_0)$--$(h_4)$ and \eqref{gtildeup},
together with \eqref{subc0}, we have that 
\cite[Proposition 4.6]{CP2} applies to functional $\J_\Omega$ written as in 
\eqref{funOm1}.
\end{proof}

\begin{remark}
In order to apply 
\cite[Proposition 4.6]{CP2} to functional $\J_\Omega$ in $X_\Omega$
condition $\tilde{G}(x,t) > 0$ for a.e. $x \in\Omega$ if $|t| \ge R$,
for a suitable $R >0$, is not necessary (such an assumption
is required in \cite{CP2}
for proving one of the geometric conditions of the Mountain Pass Theorem).
\end{remark}

For simplicity, let us assume that $0 \in \Omega$ so $\eps^*>0$
exists such that $B_{\varepsilon^*}\subset\Omega$. 

\begin{remark}  \label{RemNew}
Assume that the hypotheses of Propositions \ref{gerho} and \ref{menoinf} are
satisfied and fix $\bar{u} \in X\setminus\{0\}$ with $\supp \bar{u} \subset B_{\eps^*}$. 
Then, $\bar{u} \in X_{\Omega}$ and, from Proposition \ref{menoinf},
$\bar{s} > 0$ exists so that 
\begin{equation}\label{mp1}
\|u^*\|_{\Omega,V} > \varrho\quad\mbox{ and }
\quad\J_{\Omega}(u^*) < \alpha^*
\end{equation}
with $u^*= \bar{s}\bar{u}$ and 
$\varrho$, $\alpha^*$ as in Proposition \ref{gerho}.\\
Clearly, we have that $\supp u^* \subset B_{\eps^*}$,
so $u^* \in X_{\Omega}$, and if we consider the segment joining $0$ to $u^*$,
namely
\begin{equation}\label{segment}
\gamma^*: s\in [0,1] \mapsto s u^* \in X,
\end{equation}
we obtain that not only $\gamma^*([0,1]) \subset X_\Omega$ but also
$\supp \gamma^*(s) \subset B_{\eps^*}$ for all $s \in [0,1]$.
\end{remark}

\begin{proposition}\label{exR}
Under the assumptions in Theorem \ref{ThExist}, 
functional $\J_{\Omega}$ has at least a critical point $u_{\Omega} \in X_{\Omega}$
such that
\begin{equation}\label{mp2}
\alpha^*\ \le\ \J_{\Omega}(u_{\Omega})\ \le \ \sup_{s\in [0,1]} \J(s u^*), 
\end{equation}
with $\alpha^*$ as in Proposition \ref{gerho} and $u^*$ as in Remark \ref{RemNew}. 
\end{proposition}

\begin{proof}
From \eqref{G0=0} we have that $\J_\Omega(0)=0$. 
Moreover, from \eqref{Xequ}, \eqref{funOm} and Proposition \ref{gerho} 
we have that 
\[
u\in X_{\Omega},\quad \|u\|_{\Omega,V}= \varrho\qquad
\then\qquad \J_{\Omega}(u)\ge\alpha^*.
\]
Then, taking $u^*\in X_{\Omega}$ as in Remark \ref{RemNew}, 
from \eqref{mp1} and Proposition \ref{wCPSbdd} we have that Theorem \ref{mountainpass} applies to 
$\J_{\Omega}$ in the Banach space $X_{\Omega}$ and a critical point 
$u_{\Omega}\in X_{\Omega}$ exists such that
\[
\J_{\Omega}(u_{\Omega}) = \inf_{\gamma \in \Gamma_{\Omega}} \sup_{s\in [0,1]} \J_{\Omega}(\gamma(s)) \ge \alpha^*
\]
with $\Gamma_{\Omega} = \{ \gamma \in C([0,1],X_{\Omega}):\, \gamma(0) = 0,\; \gamma(1) = u^*\}$.\\
On the other hand, also the second inequality in \eqref{mp2} is true as, 
from Remark \ref{RemNew}, segment $\gamma^*$ in \eqref{segment} is such
that $\gamma^* \in \Gamma_{\Omega}$. 
\end{proof}

At last, let us point out that,
in the following, we need to prove a uniform boundedness 
for a sequence which is bounded in $W_0^{1,p}(\Omega)$.
Clearly, since $\Omega$ is a bounded subset of $\R^N$,
if $p > N$ from the Sobolev Embedding Theorem
such a boundedness is trivially satisfied
but, on the contrary, 
if $p \le N$ the following sufficient conditions are required. 

\begin{lemma}
\label{tecnico} 
Let $\Omega$ be an open bounded domain in $\R^N$ and
consider $p$, $q$ so that  $1 < p \le N$
and $p \le q < p^*$ (if $N = p$ we just require that $p^*$ is any
number larger than $q$)
and take $u \in W_0^{1,p}(\Omega)$. If $a^* >0$ and $m_0\in \N$
exist such that 
\begin{equation}
\label{LUuno}
\int_{\Omega^+_m}|\nabla u|^p dx \le a^* \left(m^q\ |\Omega^+_m| +
\int_{\Omega^+_m} |u|^q dx\right)\quad\hbox{for all $m \ge m_0$,}
\end{equation}
with $\Omega^+_m = \{x \in \Omega:\ u(x) > m\}$, then $\displaystyle \esssup_{\Omega} u$
is bounded from above by a positive constant which can be chosen so
that it depends only on $|\Omega^+_{m_0}|$, $N$, $p$, $q$, $a^*$, $m_0$, 
$\|u\|_{\Omega}$, or better by a positive constant which can be chosen so 
that it depends only on $N$, $p$, $q$, $a^*$, $m_0$ and $a_0^*$
for any $a_0^* > 0$ such that 
\[
\max\{|\Omega^+_{m_0}|,\ \|u\|_{\Omega}\}\ \le\ a_0^*.
\]
Vice versa, if inequality
\[
\int_{\Omega^-_m}|\nabla u|^p dx \le a^*\left(m^q\ |\Omega^-_m| +
\int_{\Omega^-_m} |u|^q dx\right) \quad\hbox{for all $m \ge m_0$,}
\]
holds, with $\Omega^-_m = \{x \in \Omega:\ u(x) < - m\}$, 
then $\displaystyle \esssup_{\Omega}(-u)$ 
is bounded from above by a positive constant which can be
chosen so that it depends only on $N$, $p$, $q$, $a^*$, $m_0$ and any
constant which is greater than both $|\Omega^-_{m_0}|$ and 
$\|u\|_{\Omega}$.
\end{lemma}

\begin{proof}
The proof is essentially as in \cite[Theorem  II.5.1]{LU} 
but, in order to point out some uniform estimates on the $L^\infty$--norms
of a sequence which is bounded in $W_0^{1,p}(\Omega)$, here 
we prefer to give some more details. \\
Firstly, taking any $m\in \N$ we note that direct
computations, the H\"older inequality and \eqref{stima*}
imply that
\begin{equation}\label{astar0}
\begin{split}
\int_{\Omega^+_m} (u-m)^q dx\ &\le 
|\Omega^+_m|^{1-\frac{q}{p^*}} \left(\int_{\Omega^+_m} (u-m)^{p^*} dx\right)^{\frac{q}{p^*}}\\
&\le \sigma_*^q |\Omega^+_m|^{1- \frac{q}{p^*}} \left(\int_{\Omega^+_m} |\nabla u|^{p} dx\right)^{\frac{q}{p}},
\end{split}
\end{equation}
with $\sigma_* > 0$ as in \eqref{stima*},
since $\max\{u-m,0\} \in W^{1,p}_0(\Omega)$. 
Thus, from \eqref{astar0} and direct computations it follows that 
\[
\begin{split}
\int_{\Omega^+_m} |u|^q dx \ &\le\ 
2^{q-1} \left(m^q |\Omega^+_m| + \int_{\Omega^+_m} (u-m)^q dx\right)\\
&\le 2^{q-1} \left(m^q |\Omega^+_m| + 
\sigma_*^q |\Omega^+_m|^{1- \frac{q}{p^*}} \left(\int_{\Omega^+_m} |\nabla u|^{p} dx\right)^{\frac{q}{p}}\right)\\
&\le\ 
2^{q-1} \left(m^q\ |\Omega^+_m| + 
\sigma_*^q |\Omega^+_{m}|^{1- \frac{q}{p^*}} \|u\|_{\Omega}^{q-p}
\int_{\Omega^+_m} |\nabla u|^{p} dx\right).
\end{split}
\]
Now, if $m \ge m_0$, from one hand \eqref{stima*} gives
\begin{equation}\label{astar1}
m^{p^*}\ |\Omega^+_m|\ \le\ \int_{\Omega^+_m} |u|^{p^*} dx\ \le\ \int_{\Omega} |u|^{p^*} dx \le L,
\end{equation}
with
\begin{equation}\label{astar11}
L^{\frac1{p^*}} = \sigma_* \|u\|_{\Omega},
\end{equation}
while, from the other hand, \eqref{LUuno} and \eqref{astar1},
\eqref{astar11} imply that
\begin{equation}\label{astar2}
\begin{split}
&\int_{\Omega^+_m}|\nabla u|^p dx \ \le\ a^*(1+2^{q-1}) m^q |\Omega^+_m|\\
&\qquad +
2^{q-1} a^* \sigma_*^q |\Omega^+_{m}|^{1-\frac{q}{p^*}} \|u\|_{\Omega}^{q-p} \int_{\Omega^+_m} |\nabla u|^{p} dx\\
&\quad \le\ a^*(1+2^{q-1}) m^q |\Omega^+_m| +
2^{q-1} a^* \sigma_*^q \frac{L^{1-\frac{q}{p^*}}}{m^{p^*-q}} \|u\|_{\Omega}^{q-p} \int_{\Omega^+_m} |\nabla u|^{p} dx\\
&\quad \le\ a^*(1+2^{q-1}) m^q |\Omega^+_m| +
2^{q-1} a^* \sigma_*^{p^*} \frac{\|u\|_{\Omega}^{p^*-p}}{m^{p^*-q}}\  \int_{\Omega^+_m} |\nabla u|^{p} dx.
\end{split}
\end{equation}
Hence, taking $m_1 \in \N$ such that
\begin{equation}\label{astar3}
m_1 \ge \max \left\{m_0, \left(2^{q} a^* \sigma_*^{p^*} \|u\|_{\Omega}^{p^*-p} \right)^{\frac{1}{p^*-q}}\right\},
\end{equation}
from \eqref{astar1}, \eqref{astar2} and direct computations it follows that
\[
\int_{\Omega^+_m}|\nabla u|^p dx \ \le\ a^*_1 m^q |\Omega^+_m|\
\le \ a^*_1 m^p\ |\Omega^+_m| \left(\frac{L}{|\Omega^+_m|}\right)^{\frac{q-p}{p^*}}
\quad\hbox{for all $m \ge m_1$,}
\]
with $a_1^* =2 a^*(1+2^{q-1})$, that is,
\[
\int_{\Omega^+_m}|\nabla u|^p dx \ \le\ a^*_1 L^{\frac{q-p}{p^*}}\ m^p\ |\Omega^+_m|^{1-\frac{p}{N}+\eps}
\quad\hbox{for all $m \ge m_1$,}
\]
where $\eps = \frac{p}{N} - \frac{q-p}{p^*} > 0$ as $q < p^*$.
Then, by using the same arguments required by estimate \eqref{astar0} but
with $q$ replaced by 1, for all $m \ge m_1$ this last inequality implies that
\[
\int_{\Omega^+_m} (u-m) dx \le \sigma_* |\Omega^+_m|^{1- \frac{1}{p^*}} \left(\int_{\Omega^+_m} |\nabla u|^{p} dx\right)^{\frac{1}{p}}
\le 
a^*_2 L^{\frac{q-p}{p p^*}}\ m\ |\Omega^+_m|^{1 + \frac{\eps}{p}}
\]
with $a_2^* = \sigma_* (a^*_1)^{\frac{1}{p}}$. Thus, reasoning as
in the proof of \cite[Lemma II.5.1]{LU} but with our setting,
from direct computations we have that 
\begin{equation}\label{astar4}
\esssup_{\Omega}u\ \le\ 2^{\frac{p}{\eps}}\ \left(m_1 + a_3^*
L^{\frac{q-p}{\eps p^*}} \int_{\Omega^+_{m_1}} |u| dx\right)
\end{equation}
with $a^*_3 = (a_2^*)^{1+\frac{p}{\eps}}$.
So, since from \eqref{astar3} and H\"older inequality 
we have that 
\[
\int_{\Omega^+_{m_1}} |u| dx \le \int_{\Omega^+_{m_0}} |u| dx \le 
|\Omega^+_{m_0}|^{1-\frac{1}{p}} 
\left(\int_{\Omega^+_{m_0}} |u|^p dx\right)^{\frac1p}
\le\ |\Omega^+_{m_0}|^{1-\frac{1}{p}} \|u\|_{\Omega},
\]
from \eqref{astar11} estimate \eqref{astar4} becomes
\begin{equation}\label{astar5}
\esssup_{\Omega}u\ \le\ 2^{\frac{p}{\eps}}\ \left(m_1 + a_3^*
\sigma_*^{\frac{q-p}{\eps}} |\Omega^+_{m_0}|^{1-\frac{1}{p}} \|u\|_{\Omega}^{1+\frac{q-p}{\eps}}\right).
\end{equation}
At last, if we take $a^*_0 \ge \max\{\|u\|_{\Omega},\ |\Omega^+_{m_0}|\}$,
estimate \eqref{astar5} implies
\[
\esssup_{\Omega}u\ \le\ 2^{\frac{p}{\eps}}\ \left(m_1^* + a_3^*
\sigma_*^{\frac{q-p}{\eps}} (a_0^*)^{2+\frac{q-p}{\eps}-\frac{1}{p}}\right)
\]
with $m_1^*$ defined as in \eqref{astar3} but replacing $\|u\|_{\Omega}$ with $a_0^*$.\\
Finally, the proof of the second statement of this lemma follows 
from the first part but applied to function $-u$.
\end{proof}


\section{Proof of the main theorem}  \label{finalsection}

Now, we are ready to prove our main result.
To this aim, we follow an approach which is similar
to that one in \cite{CPSpreprint} but, since our problem has a 
different setting and requires to rehash the proofs,
for completeness here we provide all the details.

Throughout this section, we suppose that the assumptions in Theorem \ref{ThExist} are satisfied
and for each $k\in\N$ we consider as bounded set the open ball $B_k$, its related
Banach space $X_{B_k}$ as in \eqref{Xlim} and the functional 
\[
\J_k : u \in X_{B_k} \mapsto \J_k(u) = \J_{B_k}(u) \in \R
\]
with $\J_{B_k}(u)$ defined as in \eqref{funOm}.

\begin{remark}\label{rem00}
For the sake of convenience, if $u \in X_{B_k}$ we always consider its trivial
extension as $0$ in $\R^N \setminus B_k$. Then, if we still denote such an 
extension by $u$, we have that $u \in X$, too. Moreover, 
from \eqref{G0=0}, definitions \eqref{funct} and \eqref{funOm}, 
respectively \eqref{diff} and \eqref{diffOm}, imply that
$\J_k(u) = \J(u)$, respectively
\[
\langle d\J_k(u),v\rangle = \langle d\J(u),v\rangle
\quad \hbox{for all $v \in X_{B_k}$.}
\]
\end{remark}

\begin{remark}\label{RemNew1}
If in Remark \ref{RemNew} we take $\eps^* \le 1$, then $u^* \in X_{B_1}$ and   
segment $\gamma^*$ in \eqref{segment}
is such that $\supp \gamma^*(s) \subset B_1$ for all $s \in [0,1]$. 
Hence, for all $k\in \N$
we have that $\gamma^* \in \Gamma_{B_k}$, with 
\[
\Gamma_{B_k} = \{ \gamma \in C([0,1],X_{B_k}):\, \gamma(0) = 0,\; \gamma(1) = u^*\}.
\]
Moreover, for the continuity of $\J\circ\gamma^*: s \in [0,1] \mapsto \J(s u^*) \in \R$, 
$\alpha^{**} \in \R$ exists, independent of $k$, such that
\[
\alpha^{**} = \max_{s\in [0,1]} \J(s u^*). 
\]
\end{remark}

Since for all $k \in \N$ Proposition \ref{exR}
applies to $\J_k$ in $X_{B_k}$, from Remarks \ref{rem00} and \ref{RemNew1},
a sequence $(u_k)_k \subset X$ exists such that 
for every $k \in \N$ it results:
\begin{itemize}
\item[$(i)$] $\ u_k \in X_{B_k}$ with $u_k = 0$ in $\R^N \setminus B_k$,
\item[$(ii)$] $\ \displaystyle \alpha^*\ \le\ \J(u_k)\ \le \ \alpha^{**}$,
\item[$(iii)$] $\langle d\J(u_k),v\rangle = 0$ for all $v \in X_{B_k}$,
\end{itemize}
with $\alpha^*$ as in Proposition \ref{gerho} and $\alpha^{**}$ as in Remark \ref{RemNew1}.

The proof of Theorem \ref{ThExist} requires different steps.

Firstly, we prove that sequence $(u_k)_k$
is bounded in $X$. If $p > N$, from the Sobolev Embedding Theorem 
it is enough to verify that $(\|u_k\|)_k$ is bounded
while, if $p \le N$, such a boundedness is not enough
and Lemma \ref{tecnico} needs.

From now on, we denote any strictly positive constant independent 
of $k$ by $d_i$.

\begin{proposition} \label{boundX}
A constant $M_0 > 0$ exists such that
\begin{equation}\label{bddX}
\|u_k\|_{X}\le M_0\quad\mbox{ for all } k\in\N.
\end{equation}
Moreover, for every $r \ge p$ we have also that
\begin{equation}\label{bdd2}
|u_k|_{V,r}\le M_0 \quad\mbox{ for all } k\in\N.
\end{equation}
\end{proposition}

\begin{proof}
From conditions $(i)$ and $(iii)$ we infer that
\begin{equation}\label{p0}
\left\langle d\J(u_k), u_k\right\rangle =0 \quad\mbox{ for all } k\in\N,
\end{equation}
then, taking $\mu$ as in assumption $(h_3)$, from \eqref{p0}, 
definitions \eqref{funct}, \eqref{diff}, assumptions $(h_2)$--$(h_3)$, 
$(V_1)$, $(g_3)$, \eqref{weightnorm} and direct computations we have that
\[
\begin{split}
\mu\J(u_k)\ =\  &\mu\J(u_k)-\left\langle d\J(u_k), u_k\right\rangle \\
\ge\  &\frac{\alpha_0\alpha_1}{p}\int_{\R^N}|\nabla u_k|^p dx +
\left(\frac{\mu}{p}-1\right)\int_{\R^N}V(x) |u_k|^p dx\ \ge\ d_1 \|u_k\|_{V}^p
\end{split}
\]
for a suitable constant $d_1>0$.
Whence, condition $(ii)$ implies that
\begin{equation}\label{M1}
\|u_k\|_{V}\le d_2 \quad\mbox{ for all } k\in\N
\end{equation}
with $d_2 = \left(\frac{\mu \alpha^{**}}{d_1}\right)^{\frac1p}$.\\
Now, in order to obtain estimate \eqref{bddX} if $p \le N$, from
definition \eqref{Xdefn} 
we need to prove that $(u_k)_k$ is a bounded sequence in $L^{\infty}(\R^N)$, too.
To this aim, we note that
for a fixed $k \in \N$ either $|u_k|_\infty \le 1$ or $\ |u_k|_\infty > 1$.\\
If $\ |u_k|_\infty > 1$, then 
\[
\esssup_{\R^N} u_k > 1\quad \hbox{and/or} \quad
\esssup_{\R^N} (-u_k) > 1.
\]
Assume that $\displaystyle \esssup_{\R^N} u_k > 1$
and consider the set
\[
B^{+}_{k,1} = \{x \in \R^N:\ u_k(x) > 1\}.
\]
From condition $(i)$ we have that
\[
B^{+}_{k,1} \subset \ B_k,
\]
moreover, $B^{+}_{k,1}$ is an open bounded domain
such that not only $|B^{+}_{k,1}| > 0$ but also, by means of \eqref{WVem}, it is
\[
|B^{+}_{k,1}| \le \int_{B^{+}_{k,1}} |u_k|^p dx
\le \int_{\R^N} |u_k|^p dx \le \|u_k\|^p\le d_3 \|u_k\|_V^p
\]
for a suitable $d_3>0$.
Hence, estimate \eqref{M1} gives 
\begin{equation}\label{measbdd}
|B^{+}_{k,1}|\le d_4
\end{equation}
with $d_4 = d_2^p d_3$.\\
In order to prove that Lemma \ref{tecnico} applies, for any $m \in \N$ we consider 
the new function $R^+_m : \R \to \R$ defined as
\[
t \mapsto R^+_mt = \left\{\begin{array}{ll}
0&\hbox{if $t \le m$}\\
t-m&\hbox{if $t > m$}
\end{array}\right. .
\]
Since taking $t=0$ in $(h_4)$
we obtain $\alpha_2 \le p$, for any $m \ge 1$, 
we have that $R^+_m u_k \in X_k$ and from condition $(iii)$, 
definition \eqref{diffOm} and hypotheses $(V_1)$, $(h_2)$, $(h_4)$ and 
direct computations it follows that
\[
\begin{split}
0\ =\ &\langle d\J(u_k), R^+_m u_k\rangle
=\int_{B^+_{k,m}} \left(1-\frac{m}{u_k}\right)\left(A(x, u_k)
+\frac1p A_t(x, u_k) u_k \right) |\nabla u_k|^p dx\\
&+\int_{B^+_{k,m}} \frac{m}{u_k} A(x, u_k)|\nabla u_k|^p dx
+\int_{B^+_{k,m}} \left(1-\frac{m}{u_k}\right) V(x) |u_k|^p dx-\int_{\R^N} g(x, u_k) R^+_m u_k dx\\
\ge\ 
&\frac{\alpha_0\alpha_2}{p}\int_{B^+_{k,m}} |\nabla u_k|^p dx -\int_{\R^N} g(x, u_k) R^+_m u_k dx,
\end{split}
\]
i.e.,
\begin{equation}\label{ie}
\frac{\alpha_0\alpha_2}{p}\int_{B^+_{k,m}} |\nabla u_k|^p dx
\le \int_{\R^N} g(x, u_k)R^+_m u_k dx,
\end{equation}
with $B^+_{k,m} = \{x\in\R^N: u_k(x)>m\}$. 
Clearly, it is $B^+_{k,m}\subseteq B^+_{k,1}$ so from
\eqref{measbdd} we have that
\begin{equation}   \label{OmM}
|B^+_{k,m}|\le d_4.
\end{equation}
Since $(g_3)$ implies $g(x,t) > 0$ if $t>0$ for a.e. $x \in \R^N$, then
$g(x,u_k(x)) > 0$ for a.e. $x \in B_{k,m}^{+}$; so,
from $(g_1)$, \eqref{p<q} and the Young inequality with 
conjugate exponents $\frac{q}p > 1$ and $\frac{q}{q-p}$, it follows that
\[
\begin{split}
\int_{\R^N} g(x,u_k)R^+_m u_k dx &\le 
\int_{B^{+}_{k,m}} g(x,u_k)u_k dx
\le \int_{B^{+}_{k,m}} (a_1 u_k^p + a_2 u_k^q) dx\\
&\le \int_{B^{+}_{k,m}} \left(\frac{q-p}q\ a_1^{\frac{q}{q-p}} + \frac{p}{q}\ u_k^q\right) dx 
+ a_2 \int_{B^{+}_{k,m}} u_k^q dx\\
&= \frac{q-p}q\ a_1^{\frac{q}{q-p}} |B^{+}_{k,m}|
+ \left(\frac{p}{q} + a_2\right) \int_{B^{+}_{k,m}} u_k^q dx.
\end{split}
\]
Hence, from \eqref{ie} we obtain that
\[
\int_{B^{+}_{k,m}} |\nabla u_k|^p dx \le  
a^* \left(|B^{+}_{k,m}|
+ \int_{B^{+}_{k,m}} u_k^q dx\right) \quad \hbox{for all $m \ge 1$,}
\]
with $a^*> 0$ independent of $m$ and $k$. Then, Lemma \ref{tecnico} with $\Omega=B_{k}$ 
applies and $d_5 > 1$ exists such that
\[
\esssup_{B_{k}} u_k \le d_5
\]
where, from Lemma \ref{tecnico}, \eqref{WVem}, \eqref{M1} and \eqref{OmM}
imply that constant $d_5$ can be choosen independent of $k \in \N$. \\
Similar arguments apply if $\displaystyle \esssup_{\R^N} (-u_k) > 1$,
then, summing up, $d_6 \ge 1$ exists such that
\[
|u_k|_{\infty} \le d_6\qquad \hbox{for all $k \in \N$.}
\]
So, \eqref{bddX} holds and from estimate \eqref{Lwin}
also \eqref{bdd2} is satisfied.
\end{proof}

From estimate \eqref{bddX} and \eqref{comp} it follows that $u_\infty \in W_V^{1, p}(\R^N)$
exists such that
\begin{align}\label{weakV}
&u_k\rightharpoonup u_\infty \quad\mbox{weakly in } W_V^{1,p}(\R^N),\\ \label{strong}
&u_k \to u_\infty \quad \hbox{strongly in $L^r(\R^N)$ for any $p\le r<p^*$,}\\ \label{limqo}
&u_k \to u_\infty \quad \hbox{a.e. in $\R^N$.}
\end{align}

\begin{proposition} \label{uinfty}
$u_\infty \in L^\infty(\R^N)$; hence, $u_\infty \in X$.
\end{proposition}

\begin{proof}
Taking
\[
A=\{x\in\R^N :\ \lim_{k\to+\infty} u_k(x) = u_\infty(x)\},
\]
for all $x\in A$ an integer $\nu_{x,1}\in\N$ exists such that
\[
|u_k(x) - u_\infty(x)| < 1\quad\mbox{ for all } k\ge\nu_{x,1},
\]
then from \eqref{bddX} we have that
\[
|u_\infty(x)| \le |u_\infty(x)- u_{\nu_{x,1}}(x)| +|u_{\nu_{x,1}}(x)| \le 1+ M_0.
\]
On the other hand, from \eqref{limqo} it follows 
that $|\R^N\setminus A| =0$, which completes the proof.
\end{proof}

\begin{remark}
Let $R \ge 1$ be fixed.
Actually, since $B_R$ is a bounded domain, from 
Remark \ref{Rmbdd} and \eqref{strong} it follows that
\begin{equation}	\label{strongpV}
u_{k} \to u_\infty \quad \hbox{strongly in $L_V^p(B_R)$.}
\end{equation}
\end{remark}

\begin{corollary}\label{limtau}
For all $1 \le r < +\infty$ we have that
\begin{equation}\label{strongtau}
u_k \to u_\infty \quad \hbox{strongly in $L_V^r(B_R)$}\quad\mbox{ for all } R\ge 1.
\end{equation}
\end{corollary}

\begin{proof}
Taking $1 \le r \le p$, the boundedness of $B_R$ gives 
$L^p(B_R) \hookrightarrow L^{r}(B_R)$ which, together with \eqref{LVnorm}, 
assumption $(V_3)$ and \eqref{strong}, ensures that \eqref{strongtau} holds.\\
On the other hand, if we take $r >p$, then Propositions \ref{boundX} and \ref{uinfty}, together with \eqref{strongpV},
allow us to apply Lemma \ref{immergo2}, so 
\eqref{strongtau} holds, too.
\end{proof}

\begin{proposition}\label{limwpr}
We have that
\begin{equation}\label{strongwpr}
u_k \to u_\infty \quad \hbox{strongly in $W_V^{1,p}(B_R)\quad$ for all $R \ge 1$.}
\end{equation}
\end{proposition}

\begin{proof}
For simplicity, throughout this proof we denote  
any infinitesimal sequence by $(\eps_k)_k$.\\
Fixing any $R \ge 1$, from \eqref{strongpV} it is enough to prove that
\begin{equation}\label{strong2}
|\nabla u_{k} - \nabla u_\infty|_{B_R,p} \to 0 \quad \hbox{as} \ k\to+\infty.
\end{equation}
To this aim, following an idea introduced in \cite{BMP}, 
let us consider the real map $\psi(t) = t \e^{\eta t^2}$, 
where $\eta > (\frac{\beta_2}{2\beta_1})^2$ will be fixed once
$\beta_1$, $\beta_2 > 0$ are chosen in a suitable way later on. By
definition, we have that
\begin{equation}\label{eq4}
\beta_1 \psi'(t) - \beta_2 |\psi(t)| > \frac{\beta_1} 2\qquad \hbox{for all $t \in \R$.}
\end{equation}
Defining $v_{k}=u_k - u_\infty$, limit \eqref{weakV} implies that
\begin{equation}\label{cc2}
v_{k} \rightharpoonup 0\, \hbox{weakly in $W^{1,p}_V(\R^N)$,}
\end{equation}
while from \eqref{strong}, respectively \eqref{limqo},
it follows that
\begin{equation}\label{cc21}
v_{k} \to 0 \quad \hbox{strongly in $L^p(\R^N)$,} 
\end{equation}
respectively
\begin{equation}\label{cc22}
v_{k} \to 0 \quad\hbox{a.e. in $\R^N$.}
\end{equation}
Moreover, from Proposition \ref{uinfty} and \eqref{bddX} we obtain that
\begin{equation}\label{cc23}
v_k \in X\quad \hbox{and}\quad |v_{k}|_\infty \le \bar{M}_0 \quad\hbox{for all $k \in\N$,}
\end{equation}
with $\bar{M}_0 = M_0 + |u_\infty|_\infty$.
Now, let $\chi_R \in C^{\infty}(\R^N)$ be a cut--off function such that
\begin{equation}\label{cut1}
\chi_R(x) = \left\{\begin{array}{ll}
1 &\hbox{if $|x| \le R$}\\
0 &\hbox{if $|x| \ge R + 1$}
\end{array}\right.,\qquad
\hbox{with}\quad 0 \le \chi_R(x) \le 1\quad \hbox{for all $x \in \R^N$,}
\end{equation}
and
\begin{equation}\label{cut2}
|\nabla \chi_R(x)| \le 2\quad \hbox{for all $x \in \R^N$.}
\end{equation}
Thus, for every $k \in \N$ we consider the new function
\[
w_{R,k}: x \in \R^N \mapsto w_{R,k}(x) = \chi_R(x) \psi(v_k(x)) \in \R.
\]
We note that \eqref{cc23} gives
\begin{equation}\label{stim1}
|\psi(v_{k})| \le \psi(\bar{M}_0),\quad 0<\psi'(v_{k}) \le \psi'(\bar{M}_0) \qquad\hbox{a.e. in $\R^N$,}
\end{equation}
while \eqref{cc22} implies
\begin{equation}\label{stim2}
\psi(v_{k}) \to 0, \quad
\psi'(v_{k}) \to 1 \qquad\hbox{a.e. in $\R^N$.}
\end{equation}
Then, from \eqref{cut1} and \eqref{stim1} we have that
$\supp w_{R,k} \subset \supp \chi_R \subset B_{R+1}$ and also 
\begin{equation}\label{stim3}
|w_{R,k}| \le \psi(\bar{M}_0) \qquad\hbox{a.e. in $\R^N$.}
\end{equation}
Moreover \eqref{stim2} implies that
\begin{equation}\label{stim4}
w_{R,k} \to 0 \qquad\hbox{a.e. in $\R^N$,}
\end{equation}
while from
\begin{equation}\label{grad}
\nabla w_{R,k} = \psi(v_{k}) \nabla \chi_R + \chi_R \psi'(v_{k}) \nabla v_k\qquad\hbox{a.e. in $\R^N$}
\end{equation}
and \eqref{cc23}--\eqref{stim1} it follows that $w_{R,k} \in X_{B_{R+1}}$.
Hence, for all $k \ge R+1$ we have that
\[
w_{R,k} \in X_{B_k}
\]
so $(iii)$, \eqref{diff} and \eqref{grad} imply that
\begin{equation}  \label{6.40}
\begin{split}
0\ = &\left\langle d\J(u_k), w_{R,k}\right\rangle\
=\int_{B_{R+1}}\psi(v_k) A(x,u_k) |\nabla u_k|^{p-2} \nabla u_k\cdot\nabla\chi_R dx\\
& +\int_{B_{R+1}}\chi_R \psi^{\prime}(v_k) A(x, u_k) |\nabla u_k|^{p-2} \nabla u_k\cdot\nabla v_k dx
+\frac1p \int_{B_{R+1}} A_t(x, u_k) w_{R,k} |\nabla u_k|^p dx\\
& +\int_{B_{R+1}} V(x)|u_k|^{p-2} u_k w_{R,k} dx -\int_{B_{R+1}} g(x, u_k) w_{R,k} dx.
\end{split}
\end{equation}
We note that $(g_0)$--$(g_1)$ together with \eqref{bddX}, \eqref{limqo}, \eqref{stim3}, 
\eqref{stim4} and Dominated Convergence Theorem on the 
bounded set $B_{R+1}$ imply that
\begin{equation}\label{ip2}
\int_{B_{R+1}} g(x,u_k) w_{R,k} dx \to 0.
\end{equation}
Moreover, from assumption $(h_1)$, \eqref{bddX}, \eqref{cut2} and Hölder inequality we have
that
\begin{equation}\label{d1}
\begin{split}
&\int_{B_{R+1}} |\psi(v_k) A(x, u_k) |\nabla u_k|^{p-2} \nabla u_k\cdot\nabla\chi_R| dx 
\le d_1\int_{B_{R+1}} |\psi(v_k)| |\nabla u_k|^{p-1} dx\\
&\qquad\le d_1\|u_k\|_{ B_{R+1}}^{p-1}\left(\int_{B_{R+1}}|\psi(v_k)|^p dx \right)^{\frac1p};
\end{split}
\end{equation}
now, from \eqref{stim1}, \eqref{stim2} and the boundedness of the set $B_{R+1}$, 
by means of Dominated Convergence Theorem, we obtain 
\begin{equation}  \label{puntino}
\int_{B_{R+1}} |\psi(v_k)|^p dx \to 0.
\end{equation}
Hence, estimate \eqref{bddX}, together with \eqref{d1} and \eqref{puntino}, 
guarantees that
\begin{equation}\label{uno}
\int_{B_{R+1}} \psi(v_k) A(x, u_k) |\nabla u_k|^{p-2} \nabla u_k\cdot\nabla\chi_R dx\to 0.
\end{equation}
Furthermore, from $(V_3)$, \eqref{cut1}, Hölder inequality, \eqref{bddX} and \eqref{puntino} it follows that
\begin{equation}  \label{dimenticato}
\begin{split}
\left\vert\int_{B_{R+1}} V(x) |u_k|^{p-2} u_k w_{R,k}dx\right\vert & 
\le C_{R+1}\int_{B_{R+1}} |u_k|^{p-1} |\psi(v_k)| dx\\
&\le C_{R+1} |u_k|_{B_{R+1},p}^{p-1}\left(\int_{B_{R+1}} |\psi(v_k)|^p dx\right)^{\frac1p}\to 0.
\end{split}
\end{equation}
On the other hand, \eqref{bddX}, \eqref{cut1}, assumptions $(h_1)$--$(h_2)$ 
and direct computations ensure the existence of a constant $d_2 > 0$, 
independent of $k$, such that
\[
\begin{split}
\left\vert\int_{B_{R+1}}\chi_R A_t(x, u_k) \psi(v_{k}) |\nabla u_k|^p dx\right\vert
\le\ &\frac{d_2}{\alpha_0}\int_{B_{R+1}} \chi_R A(x,u_k) |\psi(v_{k})| |\nabla u_k|^{p} dx\\
=\ &\frac{d_2}{\alpha_0}\int_{B_{R+1}} \chi_R A(x,u_k) |\psi(v_{k})| |\nabla u_k|^{p-2}\nabla u_k\cdot\nabla v_{k} dx\\
&+\frac{d_2}{\alpha_0}\int_{B_{R+1}} \chi_R A(x, u_k) |\psi(v_{k})| |\nabla u_k|^{p-2}\nabla u_k\cdot\nabla u_\infty dx.
\end{split}
\]
We note that hypothesis $(h_1)$, \eqref{bddX}, \eqref{cut1} and Hölder inequality give
\begin{equation} \label{tre3}
\begin{split}
&\left|\int_{B_{R+1}} \chi_R A(x,u_n)|\psi(v_{k})| |\nabla u_k|^{p-2}\nabla u_k\cdot\nabla u_\infty dx\right|\\
&\qquad\le d_3\left(\int_{B_{R+1}}|\psi(v_{k})|^p |\nabla u_\infty|^p dx\right)^{\frac1p}
\left(\int_{B_{R+1}} |\nabla u_k|^p dx\right)^{\frac{p-1}{p}} \\
&\qquad\le d_4\left(\int_{B_{R+1}}|\psi(v_{k})|^p |\nabla u_\infty|^p dx\right)^{\frac1p}\
\to\ 0
\end{split}
\end{equation}
as \eqref{stim1}, \eqref{stim2} and, again, Dominated Convergence Theorem imply that
\[
\int_{B_{R+1}}|\psi(v_{k})|^p |\nabla u_\infty|^p dx\to 0.
\]
Then, from estimate \eqref{6.40} and \eqref{ip2}, \eqref{uno}--\eqref{tre3}, 
we obtain that
\begin{equation}\label{epsquasi}
\varepsilon_k\ge\int_{B_{R+1}}\chi_R A(x, u_k)h_k|\nabla u_k|^{p-2}\nabla u_k\cdot\nabla v_k dx
\end{equation}
if we set
\[
h_k(x) =\psi^{\prime}(v_k(x))-\frac{d_2}{p\alpha_0} |\psi(v_k(x))| \quad\mbox{ for a.e. } x\in\R^N.
\]
Now, back to \eqref{eq4} and taking $\beta_1=1, \beta_2 =\frac{d_2}{p\alpha_0}$, 
from \eqref{stim1}, \eqref{stim2} and direct computations not only we have that
\begin{equation} \label{stima10}
h_{k}(x) \to 1\quad\hbox{a.e. in $\R^N$} \quad 
\hbox{and}\quad
|h_{k}(x)| \le \psi'(\bar{M}_0) + d_5|\psi(\bar{M}_0)|\quad\hbox{a.e. in $\R^N$,}
\end{equation}
but also
\begin{equation} \label{frac12}
h_{k}(x)>\frac12\quad\mbox{ a.e. in $\R^N$}.
\end{equation}
At last, back to \eqref{epsquasi}, direct computations give
\begin{equation}\label{quasi2}
\begin{split} 
\varepsilon_{k}\ \ge\ &\int_{B_{R+1}} \chi_R A(x, u_k) h_k |\nabla u_k|^{p-2}\nabla u_k\cdot\nabla v_{k} dx\\
=\ &\int_{B_{R+1}} \chi_R A(x,u_\infty) |\nabla u_\infty|^{p-2}\nabla u_\infty\cdot\nabla v_{k}dx\\ 
& +\int_{B_{R+1}}\chi_R(h_{k} A(x, u_k) -A(x,u_\infty))|\nabla u_\infty|^{p-2}\nabla u_\infty\cdot\nabla v_{k} dx\\ 
& +\int_{B_{R+1}} \chi_R A(x, u_k) h_k(|\nabla u_k|^{p-2} \nabla u_k -|\nabla u_\infty|^{p-2}\nabla u_\infty)\cdot\nabla v_{k} dx,
\end{split}
\end{equation}
where from \eqref{cc2} it follows that
\[ 
\int_{B_{R+1}} \chi_R A(x,u_\infty) |\nabla u_\infty|^{p-2}\nabla u_\infty \cdot\nabla v_{k}dx\ \to\ 0.
\]
Now, from Hölder inequality, we have that
\begin{equation} \label{hkDC}
\begin{split}
&\int_{B_{R+1}} \chi_R |(h_k A(x, u_k) -A(x,u_\infty))|\nabla u_\infty|^{p-2}\nabla u_\infty\cdot\nabla v_{k}| dx\\
&\quad\le\left(\int_{B_{R+1}}|h_{k} A(x, u_k) -A(x,u_\infty)|^{\frac{p}{p-1}}|\nabla u_\infty|^p dx\right)^{\frac{p-1}{p}}
\|v_{k}\|_{B_{R+1}}\ \to\ 0,
\end{split}
\end{equation}
since assumption $(h_0)$, \eqref{limqo}, \eqref{cut1} and \eqref{stima10} imply that 
\[
h_{k}A(x,u_k) -A(x,u_\infty) \to 0 \quad \hbox{a.e. in $\R^N$,}
\]
while \eqref{bddX}, \eqref{stima10}, Proposition \ref{uinfty} and $(h_1)$ give 
\[
|h_{k} A(x, u_k) -A(x,u_\infty)|^{\frac{p}{p-1}}|\nabla u_\infty|^p\le d_6 |\nabla u_\infty|^p
\quad \hbox{a.e. in $\R^N$.}
\] 
Hence, Dominated Convergence Theorem applies which, 
together with \eqref{cc21}, guarantees that the limit in \eqref{hkDC} holds.\\
Moreover, using the previous estimates in \eqref{quasi2}, from \eqref{cut1}, \
\eqref{frac12}, hypothesis $(h_2)$ and the strong convexity of the power function 
with exponent $p>1$, we have that
\[
\varepsilon_{k}\ge\frac{\alpha_0}{2}\int_{B_{R+1}} (|\nabla u_k|^{p-2}\nabla u_k 
-|\nabla u_\infty|^{p-2}\nabla u_\infty)\cdot\nabla v_{k} dx \ge 0,
\]
which implies that
\[
\int_{B_{R+1}} (|\nabla u_k|^{p-2}\nabla u_k -|\nabla u_\infty|^{p-2}\nabla u_\infty)\cdot\nabla v_{k} dx\
\to\ 0;
\]
hence, \eqref{strong2} follows from \eqref{weakV}.
\end{proof}

\begin{proposition}\label{critt}
We have that
\begin{equation}\label{crit0}
\langle d\J(u_\infty),\varphi\rangle = 0 \quad \hbox{for all $\varphi \in C_c^\infty(\R^N)$}
\end{equation}
with $C_c^\infty(\R^N) = \{\varphi \in C^\infty(\R^N):\ \supp\varphi\subset\subset \R^N\}$.
Hence, $d\J(u_\infty)= 0$ in $X$.  
\end{proposition}

\begin{proof}
Taking $\varphi \in C_c^\infty(\Omega)$, a radius $R \ge 1$ exists
such that $\supp \varphi \subset B_R$. Thus, for all $k\ge R$
we have that $\varphi \in X_k$ so $(iii)$ applies and
$\langle d\J(u_k),\varphi\rangle = 0$.
Furthermore, direct computations give
\begin{equation}\label{crit1}
\begin{split}
0\ \le\ &|\langle d\J(u_\infty),\varphi\rangle| =
|\langle d\J(u_k),\varphi\rangle - \langle d\J(u_\infty),\varphi\rangle|\\
\le\ &\int_{B_R} |A(x, u_k)|\big| |\nabla u_k|^{p-2} \nabla u_k -|\nabla u_\infty|^{p-2}\nabla u_\infty\big| |\nabla\varphi| dx\\
& +\int_{B_R}|A(x, u_k)-A(x,u_\infty)| |\nabla u_\infty|^{p-1} |\nabla \varphi| dx\\
& + \frac{|\varphi|_{\infty}}{p}\int_{B_R} |A_t(x, u_k)| \big| |\nabla u_k|^p -|\nabla u_\infty|^p\big| dx \\
& + \frac{|\varphi|_{\infty}}{p}\int_{B_R}|A_t(x, u_k) -A_t(x,u_\infty)| |\nabla u_\infty|^p dx\\
& +|\varphi|_{\infty}\int_{B_R} V(x) \big| |u_k|^{p-2} u_k -|u_\infty|^{p-2} u_\infty\big| dx + 
|\varphi|_\infty \int_{B_R}\big|g(x,u_k) - g(x,u_\infty)\big| dx.
\end{split}
\end{equation}
We note that the boundedness of $B_R$ together with $(g_0)$--$(g_1)$, 
\eqref{bddX}, \eqref{limqo} and Proposition \ref{uinfty} allow us to apply
the Dominated Convergence Theorem and 
\[
\int_{B_R}\big|g(x,u_k) - g(x,u_\infty)\big| dx \to 0.
\]
Moreover, by using some arguments similar to those ones in the proof 
of Proposition \ref{C1} we have that
assumption $(V_3)$ and \eqref{strong} imply that
\[
\int_{B_R} V(x) \big| |u_k|^{p-2} u_k -|u_\infty|^{p-2} u_\infty\big| dx\ \to\ 0.
\]
On the other hand, by means of $(h_1)$, Hölder inequality and direct computations, we obtain 
\begin{equation}\label{A1}
\begin{split}
&\int_{B_R} |A(x, u_k)|\big| |\nabla u_k|^{p-2} \nabla u_k -|\nabla u_\infty|^{p-2}\nabla u_\infty\big| |\nabla\varphi| dx\\
&\qquad\qquad\qquad\le d_1 \left(\int_{B_R}\big||\nabla u_k|^{p-2}\nabla u_k -
|\nabla u_\infty|^{p-2} u_\infty\big|^{\frac{p}{p-1}} dx\right)^{\frac{p-1}{p}},
\end{split}
\end{equation}
for a suitable constant $d_1 = d_1(|\varphi|_{\infty}) > 0$. 
Now, if $p>2$, then Lemma \ref{lemmavett}, again Hölder inequality with conjugate exponents $p-1$ 
and $\frac{p-1}{p-2}$, \eqref{bddX} and direct computations imply that
\begin{equation}\label{A2}
\left(\int_{B_R}\big||\nabla u_k|^{p-2}\nabla u_k -
|\nabla u_\infty|^{p-2} u_\infty\big|^{\frac{p}{p-1}} dx\right)^{\frac{p-1}{p}}
\le d_2 \|u_k - u_\infty\|_{B_R}.
\end{equation}
Conversely, if $1<p\le 2$ again Lemma \ref{lemmavett} applies so that
\begin{equation}\label{A3}
\left(\int_{B_R}\big||\nabla u_k|^{p-2}\nabla u_k -|\nabla u_\infty|^{p-2} u_\infty\big|^{\frac{p}{p-1}} dx\right)^{\frac{p-1}{p}}
\le \|u_k - u_\infty\|_{B_R}^{p-1}.
\end{equation}
Thus, from \eqref{A1}--\eqref{A3} and \eqref{strongwpr} it follows that
\begin{equation}\label{Aprimo}
\int_{B_R} |A(x, u_k)|\big| |\nabla u_k|^{p-2} \nabla u_k -|\nabla u_\infty|^{p-2}\nabla u_\infty\big| |\nabla\varphi| dx \
\to\ 0.
\end{equation}
Furthermore, from assumption $(h_1)$, \eqref{bddX}, \eqref{strongwpr} and Lemma \ref{lemmaOm} we have that
\[
\int_{B_R} |A_t(x, u_k)| \big| |\nabla u_k|^p -|\nabla u_\infty|^p\big| dx\ \to\ 0,
\]
while, from Dominated Convergence Theorem it results
\[
\int_{B_R}|A_t(x, u_k) -A_t(x,u_\infty)| |\nabla u_\infty|^p dx\ \to\ 0
\]
as $(h_0)$ and \eqref{strongwpr} imply that
\[
|A_t(x, u_k) -A_t(x,u_\infty)| |\nabla u_\infty|^p \to 0\quad\mbox{ a.e. in } \R^N
\]
and $(h_1)$, \eqref{bddX} give
\[
|A_t(x, u_k) -A_t(x,u_\infty)| |\nabla u_\infty|^p \le d_3 |\nabla u_\infty|^p \in L^1(B_R).
\]
So, summing up, \eqref{crit1} guarantees that \eqref{crit0} holds and 
the proof follows from 
the density of $C_c^{\infty}(\R^N)$ in $X$.
\end{proof}

\begin{proof}[Proof of Theorem \ref{ThExist}]
From Proposition \ref{critt} we have that the statement of Theorem \ref{ThExist} is true 
if we prove $u_\infty \not\equiv 0$.\\ 
Arguing by contradiction, assume that $u_\infty=0$. Hence, from assumption $(g_1)$ with $q<p^*$ and \eqref{strong}
we have that
\begin{equation}\label{gto0}
\int_{\R^N} g(x, u_k)u_k dx\ \to\ 0,
\end{equation}
which, together with assumption $(g_3)$, ensures that
\begin{equation}\label{Gto0}
\int_{\R^N} G(x, u_k) dx \to 0.
\end{equation}
Furthermore, from \eqref{diff}, $(iii)$, $(h_4)$ and also $(h_2)$ it results
\[
\begin{split}
0\ =\ &\left\langle d\J(u_k), u_k\right\rangle =\int_{\R^N} A(x, u_k)|\nabla u_k|^p dx
 +\frac1p\int_{\R^N} A_t(x, u_k) u_k |\nabla u_k|^p dx\\
& +\int_{\R^N} V(x) |u_k|^p dx -\int_{\R^N} g(x, u_k) u_k dx\\
\ge\ & \frac{\alpha_2\alpha_0}{p}\int_{\R^N} |\nabla u_k|^p dx+\int_{\R^N} V(x) |u_k|^p dx -\int_{\R^N} g(x, u_k) u_k dx,
\end{split}
\]
which implies that
\begin{equation} \label{weigh0}
\|u_k\|_V \to 0\quad\mbox{ as } k\to +\infty
\end{equation}
by means of \eqref{weightnorm} and \eqref{gto0}.
Hence, from \eqref{funct}, $(h_1)$, $\eqref{bddX}$, \eqref{Gto0} and \eqref{weigh0} we infer 
\[
\J(u_k) \le d_1 \|u_k\|_V^p -\int_{\R^N} G(x, u_k) dx \to 0
\]
in contradiction with estimate $(ii)$. 
\end{proof}

\begin{remark}
Unlike the results in \cite{BCS2015,BCS2016,AS} and other similar works, 
here managing our problem directly on $\R^N$
seems quite hard. In fact, the choice of the working space in \eqref{Xdefn} 
requires that the $(wCPS)$ condition holds if $u\in L^{\infty}(\R^N)$ too, 
but we have no knowledge of a result similar to Lemma \ref{tecnico} 
but settled in the whole Euclidean space $\R^N$.
\end{remark}


\end{document}